\documentclass{svproc}
\usepackage{amsmath}
\usepackage{mathtools}
\usepackage{amssymb}
\usepackage{hyperref}
\usepackage{graphicx}
\font\mfett=cmmib10 at10pt
\font\mfetts=cmmib10 at8pt
\def\alphar{\hbox{\mfett\char011}}
\def\deltar{\hbox{\mfett\char014}}
\def\deltas{\hbox{\mfetts\char014}}
\usepackage{url}

\begin{document}
\mainmatter              
\title{Modifications of Prony's Method for the Recovery and Sparse Approximation of  Generalized Exponential Sums}
\titlerunning{Modifications of Prony's Method}  
%
\author{Ingeborg Keller\inst{1} \and Gerlind Plonka\inst{1}
}
\authorrunning{Ingeborg Keller \and Gerlind Plonka} 
%
\tocauthor{Ingeborg Keller and Gerlind Plonka}
\institute{University of G\"ottingen, Institute for Numerical and Applied Mathematics, Lotzestra{\ss}e 16-18, 37083 G\"ottingen,  Germany,\\
\email{i.keller,plonka@math.uni-goettingen.de},\\ WWW home page:
\texttt{http://na.math.uni-goettingen.de}
}

\maketitle              

\begin{abstract}
In this survey we describe some modifications of Prony's method. In particular, we consider the recovery of  general expansions into eigenfunctions of linear differential operators  of first order and show, how these expansions can be recovered from function samples using generalized shift operators. We derive an ESPRIT-like algorithm for the generalized recovery method and show, that this approach can  be directly used to reconstruct classical exponential sums from non-equispaced data. Furthermore, we derive  a modification of Prony's method  for sparse approximation with exponential sums  which leads to a non-linear least-squares problem.

\end{abstract}
\section{Introduction: Recovery of Exponential Sums} \label{sec1}

The recovery and sparse approximation of structured functions is a fundamental problem in many areas of signal processing and engineering.  
In particular, exponential sums and their generalizations play an important role in time series analysis and in system theory \cite{Hau90,Lang00}, in the theory of annihilating filters, and for the recovery of signals with finite rate of innovation \cite{Dra07,VMB02,Poh10,Uri13,BSBV17}, as well as for linear prediction  methods \cite{Mano05,Stoica05}. 
For system reduction, Prony's method is related to the problem of low-rank approximation of structured matrices (particularly Hankel matrices) and corresponding nonlinear least-squares problems \cite{Mar18,UM14}.
There is a close relation between Prony's method and Pad\'{e} approximation \cite{Bar05,Cuyt18}. 
Exponential sums started to become more important also for sparse approximation of smooth functions, see \cite{BM05,BH05,Ha19,PP19}, and this question is closely related to approximation in Hardy spaces and the theory of Adamjan, Arov and Krein, see \cite{AAK71,ACH11,PP16}.

\subsection{The Classical Prony Method} \label{classicalpronysubsec}
A fundamental problem discussed in many papers is the recovery of exponential sums of the form 
\begin{equation} 
\label{exp}
f(x)\coloneqq \sum_{j=1}^{M} c_{j} \, {\mathrm e}^{\alpha_{j} x}= \sum_{j=1}^{M} c_{j} \, z_{j}^{x}, \qquad \textrm{with} \quad z_{j}\coloneqq {\mathrm e}^{\alpha_{j}},
\end{equation}
where the coefficients $c_{j} \in {\mathbb{C}}\setminus \{ 0 \}$  as well as the pairwise different frequency parameters  $\alpha_{j} \in {\mathbb{C}}$ ($j=1,\dots,M$) or equivalently, $z_{j} \in \mathbb{C}$ are unknown. For simplicity we assume that the number of terms $M$ is given beforehand. 
One important question appears: What information about $f$ is  needed in order to solve this recovery problem uniquely?
\smallskip

The classical Prony method uses the equidistant samples $f(0), \, f(1), \ldots ,$  $f(2M-1)$. Indeed, if we suppose that $\textrm{Im} \, \alpha_{j} $, $j=1,\dots,M$ lies in a predefined interval of length $2\pi$, as e.g. $[-\pi, \, \pi)$, these $2M$ samples are sufficient.
This can be seen as follows.

We can view $f(x)$  as the solution of a homogeneous linear difference equation of order $M$ with constant coefficients and try to identify these constant coefficients in a first step. We define the characteristic polynomial with the help of its (yet unknown) zeros  $z_{j} = {\mathrm e}^{\alpha_{j}}$,$j=1,\dots,M$, and consider its monomial representation, 
$$ p(z) \coloneqq \prod_{j=1}^{M} ( z - {\mathrm e}^{\alpha_{j}}) = z^{M} + \sum_{k=0}^{M-1} p_{k} \, z^{k}. $$
Then the coefficients $p_{k}$, $k=0, \ldots , M-1$, and $p_{M} = 1$ satisfy 
$$\sum_{k=0}^{M} p_{k}  f(k+m) = \sum_{k=0}^{M} p_{k} \sum_{j=1}^{M} c_{j}  z_{j}^{k+m} = \sum_{j=1}^{M} c_{j}  z_{j}^{m} \sum_{k=0}^{M} p_{k}  z_{j}^{k} 
= \sum_{j=1}^{M} c_{j}  z_{j}^{m}  p(z_{j}) = 0 $$
for all $m \in \mathbb{N}$. Thus the coefficients of the linear difference equation  $p_{k}$ can be computed by solving the linear system
$$ \sum_{k=0}^{M-1} p_{k} \, f(k+m) = - f(M+m), \qquad  m=0, \ldots , M-1. $$
Knowing $p(z)$, we can simply compute its zeros $z_{j} = {\mathrm e}^{\alpha_{j}}$, and in a further step the coefficients $c_{j}$, $j=1, \ldots , M$, by solving the system
$$ f(\ell ) = \sum_{j=1}^{M} \alpha_{j} \, z_{j}^{\ell}, \qquad \ell=0, \ldots , 2M-1. $$
In practice there are different numerical algorithms available for this method, which take care for the inherit numerical instability of this approach, see e.g. \cite{HS90,PT14,PT10,RK89}.
Note that for a given arbitrary vector $(f_{k})_{k=0}^{2M-1}$ the interpolation problem 
$$ f_{k} = \sum_{j=1}^{M} c_{j} \, z_{j}^{k}, \qquad k=0, \ldots, 2M-1, $$
may not  be solvable, see e.g. \cite{chunaev16}. The characteristic polynomial $p(z)$ of the homogeneous difference equation $\sum_{k=0}^{M} p_{k} f_{k+m} = 0$, $m=0, \ldots , M-1$, may have zeros with multiplicity greater than $1$, whereas the exponential sum in (\ref{exp}) is only defined for pairwise different zeros. In this paper, we will will exclude the case of zeros with multiplicity greater than $1$. However, the zeros of the characteristic polynomial  $e^{\alpha_{j}}$ resp. the parameters $\alpha_j$, $j=1,\dots,M$, may be arbitrarily close. This may lead to highly ill-conditioned  matrices.
\subsection{Content of this Paper}\label{subseccont}
In this paper, we will particularly consider the following questions.
\begin{enumerate}
\item How can we generalize Prony's method in order to recover other expansions than (\ref{exp})?
\item What kind of information is needed in order to recover the considered expansion?
\item How can we modify Prony's method such that we are able to optimally approximate a given (large) vector of function values in the Euclidean norm by a sparse exponential sum?
\end{enumerate}

To tackle  the first question, we introduce the operator based general Prony method and particularly apply it to study expansions of the form 
\begin{equation}\label{expGa}
f(x) = \sum_{j=1}^{M} c_{j} \, H(x) \, {\mathrm e}^{\alpha_{j}G(x)} , \qquad x \in [a, b] \subset {\mathbb{R}},
\end{equation}
where $c_{j}, \, \alpha_{j}\in {\mathbb{C}}$, $c_{j} \neq 0$,  $\alpha_{j}$ pairwise different, $G, H\in C^{\infty}(\mathbb{R})$ are  predefined functions, where  $G$ is strictly monotone on $[a,b]$, and $H$ is nonzero on $[a,b]$.  
This model covers many interesting examples as e.g.\ shifted Gaussians, generalized monomial sums and others.
For the expansions (\ref{expGa}) we will derive different sets of samples which are sufficient for the recovery of all  model parameters, thus answering the second question.
 
In regard to question 3 we will show for the case of $f$ as  in (\ref{exp}) and (\ref{expGa}), how the methods need to be modified for optimal approximation, and how to treat the case of noisy measurements.
\medskip

The outline of the paper is as follows. First we will introduce the idea of an operator based Prony method by looking at the recovery problem of the classical exponential sum from different angles.
In Section \ref{sec3}, we study  the recovery of the more general expansion $f$ of the form (\ref{expGa}). We will show that  (\ref{expGa}) can be viewed as  an expansion into eigenfunctions of a differential operator of first order and thus, according to the generalized Prony method in \cite{PP13}, can be recovered  using higher order derivative values of $f$.
We will show construct a new generalized shift operator which possesses the same eigenfunctions. This leads to a recovery method that requires only function values of $f$ instead of derivative values. The idea will be further illustrated with several examples in Subsection \ref{subsecexamp}.
Section \ref{sec4} is devoted to the numerical treatment of the generalized recovery method. We will derive an ESPRIT-like algorithm for the computation of all unknown parameters in the expansion (\ref{expGa}). This algorithm also applies if the  number of terms $M$ in the expansion (\ref{expGa}) is not given beforehand. Furthermore, we show in Section \ref{subsecpart}, how the recovery problem can be simplified if some frequencies $\alpha_{j}$,$j \in \{1,\dots,M\}$, are known beforehand (while the corresponding coefficients $c_{j}$ are unknown). In Section \ref{subsecnoneq}, we use a different interpretation of (\ref{expGa}) in order to
derive a new method to recover an exponential sum from non-equispaced functions samples.
Finally, in Section \ref{sec5} we study the optimal approximation with exponential sums in the Euclidean norm. This leads to a nonlinear least squares problem which we tackle directly using a Levenberg-Marquardt iteration.
Our approach is essentially different from earlier algorithms, as e.g. \cite{BM86,OS91,OS95,ZP19}.


%
\section{Operator Based View to Prony's Method}\label{sec2}
In order to tackle the questions 1 and 2 in Section \ref{subseccont}, we start by reconsidering  Prony's method.
As an introductory example, we study the exponential sum in (\ref{exp}) from a slightly different viewpoint.
For $h \in {\mathbb{R}}\setminus \{ 0 \}$ let $S_{h}: C^{\infty}(\mathbb{R}) \to C^{\infty}(\mathbb{R})$  be the shift operator given by $ S_{h} f \coloneqq f\left(\cdot + h\right)$. Then, for any  $\alpha \in \mathbb{C}$, the function ${\mathrm e}^{\alpha x }$  is an eigenfunction of $S_{h}$  with eigenvalue ${\mathrm e}^{\alpha h}$, i.e.,
$$ (S_{h} {\mathrm e}^{\alpha \cdot})(x) = {\mathrm e}^{\alpha(x+h)} = {\mathrm e}^{\alpha h}  \, {\mathrm e}^{\alpha x}. $$
Therefore, the exponential sum in (\ref{exp}) can be seen as a sparse expansion into eigenfunctions of the shift operator $S_{h}$.  The eigenvalues ${\mathrm e}^{\alpha_{j} h}$ are pairwise different, if we assume that $\textrm{Im} \, \alpha_{j} \in [-\pi/h, \, \pi/h )$. 
Now we consider the Prony polynomial 
$$ p(z) \coloneqq \prod_{j=1}^{M} \left(z-{\mathrm e}^{\alpha_{j} h}\right) = \sum_{k=0}^{M} p_{k} \, z^{k} $$
defined by the (unknown) eigenvalues ${\mathrm e}^{\alpha_{j}h}$ corresponding to the active eigenfunctions in the expansion $f$ as in (\ref{exp}).Then, for any predefined $x_{0} \in \mathbb{R}$ we have
\begin{align}\label{hh}
\sum_{k=0}^{M} p_{k} f\left(x_0 + h(k+m)\right) &= \sum_{k=0}^{M} p_{k} \left(S_{h}^{k+m} f\right)(x_{0}) = \sum_{k=0}^{M} p_{k} \sum_{j=1}^{M} c_{j}  \left(S_{h}^{k+m} {\mathrm e}^{\alpha_{j} \cdot}\right)(x_{0}) \notag  \\
&= \sum_{j=1}^{M} c_{j} \sum_{k=0}^{M} p_{k} {\mathrm e}^{\alpha_{j} (hm+hk)} \, {\mathrm e}^{\alpha_{j}x_{0}} \notag \\
&= \sum_{j=1}^{M} c_{j} {\mathrm e}^{\alpha_{j} hm} p\left({\mathrm e}^{\alpha_{j} h}\right) {\mathrm e}^{\alpha_{j}x_{0}} = 0,
\end{align}
i.e., we can reconstruct $p(z)$ by solving this homogeneous system for $m=0, \ldots , M-1$.
We conclude that the exponential sum in (\ref{exp}) can be recovered from the samples $f(h\ell+ x_{0})$, $\ell=0, \ldots , 2M-1$. This is a slight generalization of the original Prony method in section \ref{classicalpronysubsec} as we introduced  an arbitrary sampling distance $h \in \mathbb{R} \setminus\{0\}$ and a starting point $x_0 \in \mathbb{R}$. 

Moreover, we can also replace the samples $\left(S_{h}^{k+m} f\right)(x_{0})= f(h(k+m)+x_{0})$ in the above computation (\ref{hh}) by any other representation of the form $F\left(S_{h}^{k+m} f\right)$, where $F: C^{\infty}(\mathbb{R}) \to \mathbb{C}$ is a linear functional satisfying $F \left({\mathrm e}^{\alpha \cdot}\right) \neq 0$, since
$$ \sum_{k=0}^{M} p_{k} F\left(S_{h}^{k+m} f\right) = \sum_{k=0}^{M} p_{k} \sum_{j=1}^{M} c_{j}  F\left(S_{h}^{k+m} {\mathrm e}^{\alpha_{j} \cdot}\right)
= \sum_{j=1}^{M} c_{j} {\mathrm e}^{\alpha_{j} hm} p\left({\mathrm e}^{\alpha_{j} h}\right) \, F\left({\mathrm e}^{\alpha_{j} \cdot}\right) =0. $$
Any set of samples of the form $F(S_{h}^{\ell}f)$, $\ell=0, \ldots , 2M-1$, is sufficient to recover $f$ in (\ref{exp}), and the above set is obtained using the point  evaluation functional $F=F_{x_{0}}$ with $F_{x_{0}}f \coloneqq f(x_{0})$ for $x_0 \in \mathbb{R}$. 

This operator-based view leads us to the generalized Prony method introduced in \cite{PP13}, which can be applied to recover any sparse expansion into eigenfunctions of a linear operator. 

To illustrate this idea further, let us consider the differential operator $D: C^{\infty}(\mathbb{R}) \to C^{\infty }(\mathbb{R})$ with $(D f)(x) \coloneqq f'(x)$ with $f'$ denoting the first derivative of $f$. Due to 
$$ \left(D {\mathrm e}^{\alpha \cdot}\right)(x) = \alpha \, {\mathrm e}^{\alpha x} $$
we observe that exponentials ${\mathrm e}^{\alpha x}$ are eigenfunctions of $D$ for any $\alpha \in \mathbb{C}$.
Thus, the sum of exponentials in (\ref{exp}) can also be seen as a sparse expansion into eigenfunctions of the  differential operator $D$.
Similarly as before let 
$$ \widetilde{p}(z) \coloneqq \prod_{j=1}^{M} \left(z-\alpha_{j}\right) = \sum_{k=0}^{M} \widetilde{p}_{k} \, z^{k}$$ 
be the characteristic polynomial being defined by the eigenvalues $\alpha_{j}$ corresponding to the ``active'' eigenfunctions of $D$ in (\ref{exp}), where again $\widetilde{p}_{M} =1$ holds. Choosing the functional $F f \coloneqq f(x_{0})$ for some fixed $x_{0} \in \mathbb{R}$, we find
\begin{eqnarray*}
\sum_{k=0}^{M} \widetilde{p}_{k} F\left(D^{k+m} f\right) &=& \sum_{k=0}^{M} \widetilde{p}_{k} f^{(k+m)}(x_{0}) = \sum_{k=0}^{M} \widetilde{p}_{k} \sum_{j=1}^{M} c_{j} \alpha_{j}^{k+m} \, {\mathrm e}^{\alpha_{j} x_{0}}  \\
&=& \sum_{j=1}^{M} c_{j} \alpha_{j}^{m} \, \widetilde{p}\left(\alpha_{j}\right) \, {\mathrm e}^{\alpha_{j} x_{0}} =0.
\end{eqnarray*}
Thus we can determine $\widetilde{p}_j$,$j=1,\dots,M$, from the homogeneous system \\
$\sum_{k=0}^{M} \widetilde{p}_{k} f^{(k+m)}(x_{0}) =0$ for $m=0, \ldots , M-1$ and $\widetilde{p}_{M}=1$, and recover the zeros $\alpha_{j}$ of $\widetilde{p}$ in a first step. The $c_{j}$ are computed in a second step the same way as in the classical case.
We conclude that also the sample set $f^{(\ell)}(x_{0})$, $\ell=0, \ldots, 2M-1$, for any fixed value $x_{0} \in {\mathbb{R}}$, is sufficient to recover $f$. WE note that here we do not have any restrictions in regards of $\textrm{Im}\, \alpha_{j}$. 

This example already shows, that there exist many different sample sets that may be used to recover the exponential sum. In particular, each set of the form 
$F\left(A^{\ell} h\right)$, $\ell=0, \ldots , 2M-1$, where $F$ is an arbitrary (fixed) linear functional 
satisfying $F \left({\mathrm e}^{\alpha \cdot }\right)  \neq 0$ and $A: C^{\infty}(\mathbb{R}) \to C^{\infty}(\mathbb{R})$ being a linear operator with eigenfunctions ${\mathrm e}^{\alpha x}$ corresponding to pairwise different eigenvalues (at least for the range of $\alpha$ covering the $\alpha_{j}$ in (\ref{exp})) can be employed for recovery. 
\medskip

However, in practice it is usually much easier to obtain function samples of the form $f(x_{0} + h\ell)$ than higher order derivative values $f^{\ell}(x_{0})$ for $\ell =0,\dots, 2M-1$. 
Therefore, for more general expansions, for example of the form (\ref{expGa}), we will raise the following question which has also been investigated in \cite{SP19}: Suppose we already found a set of samples which is (theoretically) sufficient to recover the expansion at hand. Is it possible to find other sets of samples which can be more easily acquired and also admit a unique recovery of the sparse expansion?  
In terms of linear operators, we can reformulate this idea: Suppose we have already found an operator $A$, such that a considered expansion $f$ is a sparse expansion into $M$ eigenfunctions of $A$ (to pairweise different eigenvalues). Is it possible to find another operator $B$ that possesses the same eigenfunctions as $A$, such that the samples $\widetilde{F}(B^{\ell})f$ (with some suitable linear functional $\widetilde{F}$) can be simpler obtained than $F(A^{\ell})f$ for $\ell=0, \ldots , 2M-1$? 
\medskip

In our introductory example for the exponential sum (\ref{exp}), let the linear functional $F$ be given as $Ff \coloneqq f(0)$. 
Assume that we had found the recovery of (\ref{exp}) from the samples $f^{(\ell)} (0)$, $\ell=0, \ldots , 2M-1$ first. This sampling set corresponds   to the linear differential operator $A=D$ with $Df = f'$. 
How can we find the shift operator $B=S_{h}$, knowing just the fact,  that (\ref{exp}) can  be viewed as a sparse expansion into eigenfunctions of $D$?
Is there a simple link between 
the linear differential operator $D$ and the shift operator $S_{h}$?

This is indeed the case. Taking  $\varphi \in C^{\infty}(\mathbb{R})$ with $\varphi(x) = {\mathrm e}^{hx}$, and applying $\varphi$ (formally) to $D$. We observe for each exponential ${\mathrm e}^{\alpha x}$, $\alpha \in \mathbb{C}$, 
$$ \varphi(D) {\mathrm e}^{\alpha \cdot} = {\mathrm e}^{hD} {\mathrm e}^{\alpha \cdot} =
\sum_{\ell=0}^{\infty} \frac{h^{\ell}}{\ell !} \, D^{\ell} {\mathrm e}^{\alpha \cdot} =
\left( \sum_{\ell=0}^{\infty} \frac{h^{\ell}}{\ell !} \, \alpha^{\ell} \right) {\mathrm e}^{\alpha \cdot} = {\mathrm e}^{\alpha h} \, {\mathrm e}^{\alpha \cdot} = S_{h} {\mathrm e}^{\alpha \cdot}.
$$
Therefore, we also have $ \varphi(D) f = S_{h}f$ for $f$ in (\ref{exp}). We note that $\varphi$ also maps the eigenvalues of the differential operator onto the eigenvalues of the shift operator.  
We will use the idea to switch from differential operators to other more suitable operators in the next section in order to recover general sparse expansion

\section{Recovery of Generalized Exponential Sums}\label{sec3}

In this section we focus on the recovery of more general sparse expansions. 
Let $G: \mathbb{R} \to  \mathbb{C}$ be a given function in $C^{\infty}(\mathbb{R})$, which is strictly monotone in a given interval $[a,b] \subset \mathbb{R}$, and let $H(x): \mathbb{R} \to  \mathbb{C}$ be 
in $C^{\infty}(\mathbb{R})$ and nonzero in $[a,b]$. We consider expansions of the form 
\begin{equation}\label{expGH} f(x) = \sum_{j=1}^{M} c_{j} \, H(x) \, {\mathrm e}^{\alpha_{j} G(x)} , \qquad x \in [a,b] \subset \mathbb{R}, 
\end{equation}
with $c_{j} \in \mathbb{C}\setminus \{ 0 \}$ and pairwise different $\alpha_{j} \in \mathbb{C}$. Obviously,  (\ref{exp}) is a special case of (\ref{expGH}) with $G(x) = x$ and $H(x) \equiv 1$. In order to recover $f$, we need to identify the parameters $c_{j}$ and $\alpha_{j}$, $j=1, \ldots , M$.

\subsection{Expansion into Eigenfunctions of a Linear Differential Operator} \label{subsecdiffop}
According to our previous considerations in Section \ref{sec2}, we want to apply the so-called generalized Prony method introduced in  \cite{PP13}, where  we view  (\ref{expGH}) as an expansion into eigenfunctions of a linear operator. 
\medskip

{\bf Step 1.} First we need to find a linear operator $A$ that possesses the functions $H(x) {\mathrm e}^{\alpha_{j} G(x)}$ as eigenfunctions for any $\alpha_{j} \in \mathbb{C}$. For this purpose, let us define the functions 
\begin{equation}\label{gh} g(x) \coloneqq \frac{1}{G'(x)}, \qquad h(x) \coloneqq - g(x) \frac{H'(x)}{H(x)} = - \frac{H'(x)}{G'(x) H(x)},
\end{equation}
which are well defined on $[a,b]$, since $G'$ and $H$ have no zeros in $[a,b]$.  Then the operator $A: C^{\infty}(\mathbb{R}) \to C^{\infty}(\mathbb{R})$ with 
\begin{equation}\label{diff}
A f (x):  = g(x) f'(x) + h(x) f(x)
\end{equation}
 satisfies 
\begin{eqnarray*} 
 A \left( H(\cdot) {\mathrm e}^{\alpha_{j} G(\cdot)} \right) (x) &=&
g(x) \left( \alpha_{j} G'(x) H(x)  + 
H'(x) \right) {\mathrm e}^{\alpha_{j} G(x)}  + h(x) \, H(x) \, {\mathrm e}^{\alpha_{j} G(x)} \\
&=&  \alpha_{j} \, H(x) \, {\mathrm e}^{\alpha_{j} G(x)}, \,\,\, ~~~~ \alpha_j \in \mathbb{C}
\end{eqnarray*}
i.e., the differential operator $A$ indeed possesses the eigenfunctions $H(x) \, {\mathrm e}^{\alpha_{j} G(x)}$ with corresponding eigenvalues $\alpha_{j}$.
\medskip

{\bf Step 2.} To reconstruct $f$, we can apply a similar procedure as in Section \ref{sec2}. Let 
\begin{equation}\label{widetilde}
\widetilde{p}(z) \coloneqq \prod_{j=1}^{M} (z-\alpha_{j}) = \sum_{k=0}^{M} \widetilde{p}_{k} \, z^{k}, \qquad \widetilde{p}_{M} = 1, 
\end{equation}
be the characteristic polynomial defined by the (unknown) eigenvalues $\alpha_{j}$ that correspond to the active eigenfunctions of the operator $A$ as in (\ref{diff}). Let $F \colon C^{\infty}(\mathbb{R}) \to {\mathbb{C}}$ be the point evaluation functional $F f\coloneqq f(x_0)$ with $x_0 \in [a,b]$, such that $H(x_0) \neq 0$ and $G'(x_0) \neq 0$. Then we observe (\ref{expGH}) yields
\begin{eqnarray*}
\sum_{k=0}^{M} \widetilde{p}_{k} \, F(A^{m+k} f) &=& \sum_{k=0}^{M} \widetilde{p}_{k} \, \sum_{j=1}^{M} c_{j} \, F \left( A^{k+m} \left( H( \cdot) \, {\mathrm e}^{\alpha_{j} G(\cdot)} \right) \right)\\
 &=& \sum_{k=0}^{M} \widetilde{p}_{k} \, \sum_{j=1}^{M} c_{j} \, \alpha_{j}^{k+m} \, F\left( H(\cdot) \, {\mathrm e}^{\alpha_{j} G(\cdot)} \right)\\
&=& \sum_{j=1}^{M} c_{j}  \, \alpha_{j}^{m} \left( \sum_{k=0}^{M} \widetilde{p}_{k}\,  \alpha_{j}^{k} \right) \left(H(x_0) \, {\mathrm e}^{\alpha_{j} G(x_0)}\right) = 0
\end{eqnarray*}
for all integers $m \ge 0$. Thus we can compute the coefficients $\widetilde{p}_{k}$, $k=0, \ldots, M-1$, using the values $F(A^{\ell}f)$, $\ell=0, \ldots, 2M-1$. Having determined the polynomial $\widetilde{p}(z)$, we can compute its zeros $\alpha_{j}$, and afterwards solve a linear equation system to reconstruct the complex coefficients $c_{j}$ in (\ref{expGH}). 

However, the question remains, how to obtain the needed data $F(A^{\ell}f)$, $\ell=0, \ldots, 2M-1$.
We obtain 
\begin{align}\label{a1}
F(A^{0}f) &= f(x_0), \\ \nonumber
F(A^{1}f)&= g(x_0)f'x_0) +h(x_0)f(x_0),\\ \nonumber
F(A^{2}f) &= g(x_0)^{2} f''(x_0) + [g(x_0)g'(x_0) + 2g(x_0)h(x_0)]f'(x_0)\\ \nonumber
&+[g(x_0)h'(x_0)+h(x_0)^{2}] f(x_0).  \nonumber
\end{align}
Since $g$ and $h$ (and their derivatives) are known beforehand, it is sufficient to provide the first $2M$ derivative values  of $f$ at one point $x_0 \in [a,b]$ in order to reconstruct $f$. Therefore we can conclude
\begin{theorem}\label{theo1}
Let $G, \, H \in C^{\infty}([a,b])$, such that $G'$ and $H$ have no zeros on $[a,b]$, and let $x_0 \in [a,b]$ be fixed.
Then   $f$ in $(\ref{expGH})$ can be viewed as an expansion into eigenfunctions of the differential operator $A$ as in (\ref{diff}), and 
 $f$ as in $(\ref{expGH})$ can be uniquely reconstructed from the derivative samples $f^{(\ell)}(x_0)$, $\ell=0, \ldots , 2M-1$.  
\end{theorem}

\begin{proof}
As seen from the above computations the operator $A$ of the form (\ref{diff}) indeed possesses the eigenfunctions $H(x) \, {\mathrm e}^{\alpha_{j} G(x)}$. 
In order to reconstruct the parameters $\alpha_{j}$, we first have to compute the required values $F(A^{\ell}f) = (A^{\ell} f)(x_0)$, $\ell=0, \ldots , 2M-1$. For this purpose, we need to determine the triangular matrix ${\mathbf L} = (\lambda_{m,\ell})_{m,\ell=0}^{2M-1} \in \mathbb{R}^{2M \times 2M}$ such that 
$$ \left(F(A^{\ell}f) \right)_{\ell=0}^{2M-1} = \left((A^{\ell}f)(x_0) \right)_{\ell=0}^{2M-1} = {\mathbf L}\, \left( f^{(\ell)}(x_0) \right)_{\ell=0}^{2M-1}. $$
As seen in (\ref{a1}), we have already $\lambda_{0,0} \coloneqq 1$, $\lambda_{1,0} \coloneqq g(x_0), \, \lambda_{1,1} \coloneqq h(x_0)$.
In order to obtain the entries of ${\mathbf L}$, we have to consider the elements $\lambda_{m,\ell}$ as functions in $x$, starting with $\lambda_{0,0}(x) \equiv 1$. 
Induction and 
$$A^{\ell}f(x) = \sum_{r=0}^{\ell} \lambda_{\ell,r}(x) \, f^{(r)}(x)$$  
yield 
\begin{eqnarray*} A^{\ell+1} f(x)  &=& \sum_{r=0}^{\ell} g(x) \left( \lambda_{\ell,r}'(x) \, f^{(r)}(x) + \lambda_{\ell,r}(x) f^{(r+1)}(x) \right) + h(x) \, \lambda_{\ell,r}(x) \, f^{(r)}(x) \\ 
&=& \sum_{r=0}^{\ell} \left( g(x) \, \lambda_{\ell,r}'(x) + h(x) \, \lambda_{\ell,r}(x)\right) \, f^{(r)}(x) + g(x) \, \lambda_{\ell,r}(x) f^{(r+1)}(x).
\end{eqnarray*}
We can conclude the recursion 
$$ \lambda_{\ell+1,r}(x) \coloneqq \left\{ \begin{array}{ll} 
g(x) \, \lambda_{\ell,r}'(x) + h(x) \, \lambda_{\ell,r}(x) & \, ~~~~~ r=0, \ldots , \ell,\\
g(x) \, \lambda_{\ell,r}(x)  & \, ~~~~~   r=\ell+1.
\end{array} \right.
$$
The matrix entries $\lambda_{\ell,k}\coloneqq \lambda_{\ell,k}(x_0)$ are well-defined by assumption on $H$ and $G$.
In a second step, we solve the homogeneous equation system 
$$ \sum_{k=0}^{M} \widetilde{p}_{k} \, F(A^{k+m}f) =0, \qquad m=0, \ldots , M-1. $$
Then we can determine the characteristic polynomial $\widetilde{p}$ in (\ref{widetilde}) and extract its zeros $\alpha_{j}$. Finally, the coefficients $c_{j}$
can be computed from the linear system
$$  F(A^{\ell}f) = (A^{\ell}f)(x_0) = \sum_{j=1}^{M} c_{j} \, (A^{\ell} (H(\cdot) \, {\mathrm e}^{\alpha_{j} G(\cdot)}))(x_0) 
= H(x_0) \, \sum_{j=1}^{M} c_{j} \alpha_{j}^{\ell} {\mathrm e}^{\alpha_{j} G(x_0)} $$
for $\ell=0, \ldots , 2M-1$. \hfill \qed
\end{proof} 
However, the values $f^{(r)}(x_0)$, $r=0, \ldots , 2M-1$,  may not be easily accessible , and we require some extra effort to compute $F(A^{\ell}f)$.

\subsection{Expansion into Eigenfunctions of a Generalized Shift Operator}\label{subsecshift}
Our goal is to find a different set of sample values for the reconstruction of $f$ as in   (\ref{expGH}), which is easier to obtain but also sufficient for a unique reconstruction. Thus we need to find an operator $B$ which possesses the same eigenfunction as $A$ as in (\ref{diff}). In addition $B$ should satisfy that $F(B^{\ell}f)$ (with some point evaluation functions $F$) can be obtained from function values of $f$. Similarly as in Section \ref{sec2}, we consider the linear operator $B=\varphi(A) = \exp(hA)$ with $A$ as in (\ref{diff})  and $h \in {\mathbb{R}} \setminus \{ 0 \}$. We observe  for $f$ in (\ref{expGH}), 
\begin{eqnarray*}
\exp(h A) f &=& \sum_{\ell=0}^{\infty} \frac{h^{\ell}}{\ell!} \, A^{\ell} f = \sum_{\ell=0}^{\infty} \frac{h^{\ell}}{\ell!} \sum_{j=1}^{M} c_{j} \,  A^{\ell} \left( {H(\cdot)} {\mathrm e}^{\alpha_{j} G(\cdot)}\right) \\
&=& \sum_{\ell=0}^{\infty} \frac{h^{\ell}}{\ell!} \sum_{j=1}^{M} c_{j} \alpha_{j}^{\ell} \left(H(\cdot) \, {\mathrm e}^{\alpha_{j} G(\cdot)} \right)
= \sum_{j=1}^{M} c_{j} \left( \sum_{\ell=0}^{\infty} \frac{h^{\ell}}{\ell!}  \alpha_{j}^{\ell} \right) \left(H(\cdot) \, {\mathrm e}^{\alpha_{j} G(\cdot)}\right)\\
&=& \sum_{j=1}^{M} c_{j} \, {\mathrm e}^{\alpha_{j}h} \left(H(\cdot) \, {\mathrm e}^{\alpha_{j} G(\cdot)}\right) 
= H(\cdot) \, \sum_{j=1}^{M} c_{j} \, {\mathrm e}^{\alpha_{j} (h +  G(\cdot))} \\
&=& H(\cdot) \,  \sum_{j=1}^{M} c_{j} \, {\mathrm e}^{\alpha_{j} G( G^{-1} (h +  G(\cdot)))}\\
&=& \frac{H(\cdot)}{H(G^{-1}(h+G(\cdot)))}  \,  \sum_{j=1}^{M} c_{j}  H(G^{-1}(h+G(\cdot))) \, {\mathrm e}^{\alpha_{j} G( G^{-1} (h +  G(\cdot)))}\\
&=& \frac{H(\cdot)}{H(G^{-1}(h+G(\cdot)))}  \, f\left(  G^{-1} (h +  G(\cdot))\right).
\end{eqnarray*}
We therefore we define the generalized shift operator 
\begin{equation}
\label{shift}
S_{H,G,h}  f (x) \coloneqq  \frac{H(x)}{H(G^{-1}(h+G(x)))} f\left(  G^{-1} (h +  G(x))\right)
\end{equation}
with functions $G, \, H$, and the step size $h \in {\mathbb{R}} \setminus \{ 0 \}$.
This shift operator has also been introduced in \cite{PSK19}. In particular it satisfies the properties
$$ S_{H,G,h_{2}} \left( S_{H,G,h_{1}}  f \right) = S_{H,G,h_{1}} \left( S_{H,G,h_{2}}  f \right) = S_{H,G,h_{1}+h_{2}} f $$
for all $h_{1}, h_{2} \in \mathbb{R}$, and 
\begin{equation}\label{shift1}
S_{H,G,h}^{k} f = S_{H,G,kh} f
\end{equation}
for $k \in \mathbb{Z}$, see Theorem 2.1 in \cite{PSK19}. Observe that here we do not to assume that $G$ and $H$ are $C^{\infty}(\mathbb{R})$ functions, and it is sufficient to consider continuous functions.  We only need to ensure the existence of $G^{-1}$ and $1/H$ within the considered sampling interval.
We summarize this in the following theorem. 

\begin{theorem} \label{theo2}
Let $G, \, H$ be continuous functions on an interval $[a,b]$, such that $G$ is strictly monotone in $[a,b]$  and $H$ has no zeros in $[a,b]$. Assume that the pairwise different parameters $\alpha_{j}$ in the expansion 
\begin{equation}\label{f2} f(x) = \sum_{j=1}^{M} c_{j} \, H(x) \, {\mathrm e}^{\alpha_{j} G(x)} , \qquad x \in [a,b] \subset \mathbb{R}, 
\end{equation}
 satisfy $\mathrm{Im} \, \alpha_{j} \in (-T,T]  $ and that $c_{j} \in {\mathbb{C}}\setminus \{ 0 \}$. Then $f$  can be uniquely reconstructed from the sample values $f(G^{-1}(h\ell + G(x_0)))$, $\ell =0, \ldots , 2M-1$, where $h$ is chosen such that $0 < |h| < \min \left\{\frac{\pi}{T}, \frac{|G(b)-G(a)|}{2M}\right\}$ and $\rm{sign} \,  h = \rm{sign}\,  (G(b)-G(a))$ and $x_0 \in \mathbb{R}$. 
\end{theorem}
\begin{proof} 
From the arguments above, we can conclude that $H(x) \, {\mathrm e}^{\alpha_{j} G(x)}$ is an eigenfunction of the generalized shift operator $S_{H,G,h}$ as  in (\ref{shift}) with the eigenvalue ${\mathrm e}^{\alpha_{j}h}$ ($\alpha_{j} \in \mathbb{C}$), since
\begin{eqnarray*}
S_{H,G,h} \left (H(\cdot)  {\mathrm e}^{\alpha_{j} G(\cdot)}\right) \!
&=& \!\frac{H(\cdot)}{H(G^{-1}(h+G(\cdot)))} \left( H(G^{-1}(h+G(\cdot)))  {\mathrm e}^{\alpha_{j} G(G^{-1}(h+G(\cdot)))}\!\right) \\
&=& H(\cdot) \, {\mathrm e}^{\alpha_{j}(h+G(\cdot))} = {\mathrm e}^{\alpha_{j}h} \, H(\cdot) \, {\mathrm e}^{\alpha_{j}G(\cdot)}
\end{eqnarray*}
holds.
Further, for $\textrm{Im} \, \alpha_{j} \in (-T,T]$, and $0 < |h| < \frac{\pi}{T}$, the eigenvalues ${\mathrm e}^{\alpha_{j}h}$ corresponding to active eigenfunctions in (\ref{expGH}) are pairwise different, such that we can uniquely derive the ``active'' eigenfunctions $H(x) {\mathrm e}^{\alpha_{j} G(x)}$ in (\ref{f2}) from the corresponding ``active'' eigenvalues.
We define the chacteristic polynomial 
\begin{equation}\label{pro}
p(z) \coloneqq \prod_{j=1}^{M} (z - {\mathrm e}^{\alpha_{j} h } ) = \sum_{k=0}^{M} p_{k} \, z^{k} \quad \textrm {with} \quad p_{M} = 1,
\end{equation}
using the (unknown) eigenvalues ${\mathrm e}^{\alpha_{j}h}$, where $p_{k}$, $k=0, \ldots , M-1$,  are the (unknown) coefficients of the monomial representation of $p(z)$. 
Then, we conclude
\begin{eqnarray}\nonumber
\sum_{k=0}^{M} p_{k} \,  (S_{H,G,h}^{k+m}f)(x_0) &=& \sum_{k=0}^{M} p_{k} \sum_{j=1}^{M} c_{j} \, (S_{H,G,h}^{k+m} H(\cdot) \, {\mathrm e}^{\alpha_{j} G(\cdot)})(x_0) \\
\nonumber
&=& \sum_{k=0}^{M} p_{k} \sum_{j=1}^{M} c_{j} \, {\mathrm e}^{\alpha_{j}h (k+m)} H(x_0) \, {\mathrm e}^{\alpha_{j} G(x_0)} \\
\nonumber
&=& H(x_0) \sum_{j=1}^{M} c_{j} \, {\mathrm e}^{\alpha_{j}hm} \,  {\mathrm e}^{\alpha_{j} G(x_0)} \sum_{k=0}^{M} p_{k} \, ({\mathrm e}^{\alpha_{j}h})^{k} \\
\label{p0}
&=& H(x_0) \sum_{j=1}^{M} c_{j} \, {\mathrm e}^{\alpha_{j}hm} \,  {\mathrm e}^{\alpha_{j} G(x_0)} \, p({\mathrm e}^{\alpha_{j}h}) = 0
\end{eqnarray}
for all integers $m$, where by definition
$$ (S_{H,G,h}^{k+m}f)(x_0) = \frac{H(x_0)}{H(G^{-1}(h(k+m) + G(x_0)))} f(G^{-1}(h(k+m)+G(x_0))). $$ 
Thus, we can compute the coefficients from the homogeneous linear system  $p_{k}$, $k=0, \ldots , M-1$ 
$$ 
\sum_{k=0}^{M} p_{k} \, (S_{H,G,h}^{k+m}f)(x_0) = H(x_0) \sum_{k=0}^{M} p_{k} \, 
\frac{f(G^{-1}(h(k+m)+G(x_0)))}{H(G^{-1}(h(k+m)+G(x_0)))} = 0,   $$
for $m=0, \ldots , M-1$, and $p_{M} = 1$, or equivalently from 
\begin{equation}\label{lini} \sum_{k=0}^{M-1} p_{k} \, 
\frac{f(G^{-1}(h(k+m)+G(x_0)))}{H(G^{-1}(h(k+m)+G(x_0)))} = - \frac{f(G^{-1}(h(M+m)+G(x_0)))}{H(G^{-1}(h(M+m)+G(x_0))}, 
\end{equation}
for $ m=0, \ldots , M-1$. The conditions on $h$ in the theorem ensure that we only use samples of $f$ in $[a,b]$.
The equation system (\ref{lini}) is always uniquely solvable, since the coefficient matrix is invertible. More exactly, we have for $f$ in (\ref{f2}),
\begin{eqnarray} \nonumber
&& \left( \frac{f(G^{-1}(h(k+m)+G(x_0)))}{H(G^{-1}(h(k+m)+G(x_0)))} \right)_{m,k=0}^{M-1} = \left( \sum_{j=1}^{M} c_{j} \, {\mathrm e}^{\alpha_{j}(h(k+m)+G(x_0))) } \right)_{m,k=0}^{M-1} \\
\label{fak}
&=& \left( {\mathrm e}^{\alpha_{j}hm} \right)_{m=0,j=1}^{M-1,M} \, \textrm{diag} \left( c_{1}{\mathrm e}^{\alpha_{1}G(x_0)}, \ldots , c_{M}{\mathrm e}^{\alpha_{M}G(x_0)}\right) \, \left( {\mathrm e}^{\alpha_{j}hk} \right)_{j=1,m=0}^{M,M-1}.
\end{eqnarray}
The first and the last matrix factor are invertible Vandermonde matrices with pairwise different nodes ${\mathrm e}^{\alpha_{j}h}$, and the diagonal matrix is invertible, since $c_{j} \neq 0$.

Having solved (\ref{lini}), we can reconstruct $p(z)$ and extract all its zeros $z_{j} = {\mathrm e}^{\alpha_{j}h}$. In a second step we can compute the  coefficients $c_{j}$ from the overdetermined system 
\begin{equation} \label{cj}
f(G^{-1}(h\ell + G(x_0)) = \sum_{j=1}^{M} c_{j} \, H(G^{-1}(h\ell + G(x_0))) \, {\mathrm e}^{\alpha_{j}(h\ell + G(x_0))}, 
\end{equation}
for $\ell=0, \ldots , 2M-1$. 
\hfill \qed
\end{proof}

\subsection{Application to Special Expansions} \label{subsecexamp}

The model (\ref{expGH}) covers many special expansions, and we want to illustrate some of them.

\subsubsection{Classical Exponential Sums.}\label{subsecclex}
Obviously, the model (\ref{exp}) is a special case of (\ref{expGH}) with $G(x)\coloneqq x$ and $H(x)\coloneqq 1$.
In this case, we have
$$ g(x) \equiv 1, \quad h(x) \equiv 0$$
in (\ref{gh}) such that $A$ in (\ref{diff}) reduces to $Af = f'$. The generalized shift operator in (\ref{shift}) with $G^{-1}(x) = x$ is of the form
$ S_{1,x,h} f(x)=  f(h+x)$ and is therefore just the usual shift operator $S_{h}$ in Section \ref{sec2}.
By Theorem 1, the sample values $f^{(\ell)}(x_0)$, $\ell =0, \ldots ,2M-1$  are sufficient for recovery of $f$, where in this case the interval $[a,b]$ can be chosen arbitrarily in $\mathbb{R}$ and thus also $x_0$. Theorem 2 provides the set of sample values $f(a+h\ell)$ similarly as we had seen already in Section \ref{sec2}.

\subsubsection{Expansions into Shifted Gaussians.}\label{subsecgauss}
We want to reconstruct expansions of the form
\begin{equation}\label{gauss} f(x) = \sum_{j=1}^{M} c_{j} \, {\mathrm e}^{-\beta(x-\alpha_{j})^{2}}, 
\end{equation}
where $\beta  \in \mathbb{C} \setminus \{ 0 \} $ is known beforehand, and we need to find $c_{j} \in \mathbb{C} \setminus \{ 0 \}$ and pairwise different $\alpha_{j} \in \mathbb{C}$, see also \cite{VMB02,PSK19}. 

First, we observe 
that by 
$$ {\mathrm e}^{-\beta(x-\alpha_{j})^{2}} = {\mathrm e}^{-\beta \alpha_{j}^{2}} \,  {\mathrm e}^{-\beta x^{2}} \,  {\mathrm e}^{2\beta \alpha_{j} x},
 $$
that  these functions are  of the form $H(x) \, {\mathrm e}^{\alpha_{j} G(x)}$, where here
 $$ H(x)\coloneqq {\mathrm e}^{-\beta \alpha_{j}^{2}} {\mathrm e}^{-\beta x^{2}}, \qquad  G(x)\coloneqq 2\beta x. $$
 Using the results in Section \ref{subsecdiffop} and \ref{subsecshift}, (\ref{gh}) yields
 $$ g(x) = \frac{1}{G'(x)} = \frac{1}{2\beta}, \qquad h(x) = - g(x) \frac{H'(x)}{H(x)} = -\frac{1}{2\beta} (- 2\beta x) = x. $$
We define the operator $A$ by $Af (x)\coloneqq \frac{1}{2\beta} f'(x)  +x \, f(x)$ and find
$$
A \, \left({\mathrm e}^{-\beta(\cdot -\alpha_{j})^{2}}\right) (x)  = \left( \frac{1}{2\beta} ( - 2 \beta (x -{\alpha}_{j})) + x \right) {\mathrm e}^{-\beta(x-\alpha_{j})^{2}} = \alpha_{j} \, {\mathrm e}^{-\beta(x-\alpha_{j})^{2}} .
$$
Thus,  we can reconstruct $f$ in (\ref{gauss}) according to Theorem \ref{theo1}  from the derivative samples  $f^{(\ell)}(x_0)$, $\ell =0, \ldots , 2M-1$. Here, $x_0$ can be chosen arbitrarily in $\mathbb{R}$, since $G'(x) = 2 \beta  \neq 0$ and $H(x) \neq 0$ for all $x \in \mathbb{R}$, which means that the interval $[a,b]$ can be chosen arbitrarily in Theorem \ref{theo1}. 
\smallskip

Another sampling set is obtained  by Theorem \ref{theo2}.  We find the generalized shift operator $S_{H,G,h}$ in (\ref{shift}) here of the form 
\begin{equation}\label{SG} S_{H,G,h} f (x) = \frac{{\mathrm e}^{-\beta x^{2}}}{{\mathrm e}^{-\beta((h+2\beta x)/2\beta)^{2}}} \, f\left(\frac{h+2\beta x}{2\beta} \right) 
= {\mathrm e}^{h(x+h/4\beta)} \, f\left(x+\frac{h}{2\beta}\right).
\end{equation}
Then
\begin{eqnarray*}
S_{H,G,h} ({\mathrm e}^{-\beta(\cdot - \alpha_{j})^{2}})(x) &=& 
{\mathrm e}^{h(x+h/4\beta)} \, {\mathrm e}^{-\beta(x +\frac{h}{2\beta}- \alpha_{j})^{2}} \\
&=& {\mathrm e}^{h \alpha_{j}} \, {\mathrm e}^{-\beta(x - \alpha_{j})^{2}}.
\end{eqnarray*}
Therefore, the expansion in (\ref{gauss}) is an expansion into eigenfunctions of the generalized  shift operator in (\ref{SG})  and can be reconstructed from the equidistant samples 
$$ f\left(a + \frac{h\ell}{2\beta} \right), \qquad \ell=0, \ldots , 2M-1, $$
where $a \in \mathbb{R}$ can be chosen arbitrarily and $0 < |h| < \frac{\pi}{T}$, where $T$ is the  a priori known bound satisfying $|\alpha_{j}| < T$ for all $j=1, \ldots , M$. Since the interval $[a,b]$ occurring in Theorem \ref{theo2} can be chosen arbitrarily large, we can always choose it such that 
$$ \frac{|G(b)- G(a)|}{2M} = \frac{2|\beta| (b-a) }{2 M } > \frac{\pi}{T}.$$
Thus, there is no further condition on the choice of $h$. 
We note that it is also possible to choose $G(x)=x$ and thus substituting $\widetilde{\alpha}_j=\alpha_j 2\beta$ for $j=1,\dots,M$. This is useful in the case of $\mathrm{Im}\beta \neq 0$. 

\begin{remark}
The model (\ref{gauss}) particularly also includes expansions into modulated shifted Gaussians 
$$ f(x) = \sum_{j=1}^{M} c_{j} \, {\mathrm e}^{2\pi {\mathrm i} x \kappa_{j}} \, {\mathrm e}^{-\beta(x-s_{j})}$$
with $\kappa_{j} \in [0,1)$ and $s_{j} \in \mathbb{R}$ which have been considered in \cite{PSK19}. Since
$$ {\mathrm e}^{2\pi {\mathrm i} x \kappa_{j}} \, {\mathrm e}^{-\beta(x-s_{j})} = {\mathrm e}^{-\beta s_{j}^{2}} \, {\mathrm e}^{-\beta x^{2}} \, 
{\mathrm e}^{-x(2\beta s_{j} + 2\pi {\mathrm i} \kappa_{j})}, $$
we choose $\alpha_{j} \coloneqq 2\beta s_{j} +  2\pi {\mathrm i} \kappa_{j}$.
It is sufficient to recover the parameter $\alpha_{j}$ is sufficient to find the parameters $s_{j}$ and $\kappa_{j}$ from the real and the imaginary part of $\alpha_{j}$, respectively for $j=1,\dots,M$. 
\end{remark}

\begin{example}\label{ex:gauss}
We illustrate the recovery of expansions into shifted Gaussians and consider $f(x)$ of the form (\ref{gauss}) with $M=10$.
The parameters in Table \ref{tabelle1} have been obtained by applying uniform random sampling from the intervals $(-3,3) + {\mathrm i}(-2,2)$ for $c_{j}$ and from $(-2, 2)$ for $\alpha_{j}$. We chose the starting parameter $x_0=-1$, the step size $h=1$ and $\beta=\mathrm{i}$.    
\medskip

\begin{table}[h] \label{tabelle1}
{\small
\begin{tabular}{c|r|r|r|r|r|r|r|r|r|r}
 & $j=1$ & $j=2$& $j=3$& $j=4$ & $j=5$ &$j=6$ &$j=7$&$j=8$ &$j=9$ &$j=10$ \\[1ex] \hline
 $\textup{Re}\, c_j$ & $  -1.754$ & $-1.193 $ & $0.174$ & $-1.617$ & $2.066$  & $-1.831$ & $-1.644$ & $-1.976$ &$-1.634 $& $-0.386$ \\[1ex]
  $\textup{Im}\,  c_j $ & $ - 0.756$ & $1.694$& $- 0.279$ &$  -1.261$ & $1.620$ & $1.919$ & $ -0.245$ & $ -1.556$ & $-0.968$ &$ -0.365$ \\[1ex]
 $\alpha_j$  &  $0.380$ & $-0.951$ & $0.411$ & $0.845$ & $-1.113$ & $-1.530$ & $-0.813$ & $-0.725$ & $-0.303$ & $-0.031$\\
 \end{tabular}
 \medskip
 
 \caption{Parameters $c_{j}$ and $\alpha_{j}$ for $f(x)$ in (\ref{gauss}) with $M=10$, see Figure \ref{fig:gauss1}.}
} 
\end{table}

The reconstruction algorithm uses the $20$ samples $f(k)$, $k=-1, \ldots , 18$, which are represented as black dots in Figure \ref{fig:gauss1}.
The maximal reconstruction error for the parameters $\alpha_{j}$ parameters $c_{j}$ are
$$err_{\alpha}=1.518622755454592\cdot 10^{-11}, \qquad err_{c}=5.286537816367291\cdot 10^{-10}.$$

\begin{figure}
  \hspace*{-15mm}
      \includegraphics[width=1.2\textwidth]{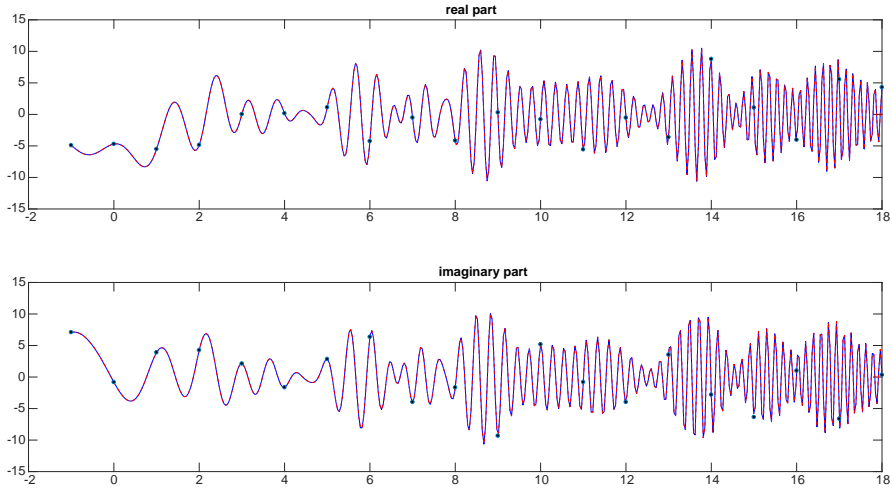} 
      \caption{Real and imaginary part of the signal $f(x)$ consisting of shifted Gaussians as given in Example \textup{\ref{ex:gauss}}. The black dots indicate the used signal values. Here the reconstructed signal is shown in red.}
  \label{fig:gauss1}
\end{figure}
\end{example}

\subsubsection{Expansions into Functions of the Form $\exp( \alpha_{j}\sin x)$.}\label{subsecexpsin}

We want to reconstruct  expansions of the form 
\begin{equation}\label{esin}
f(x) = \sum_{j=1}^{M} c_{j} \, {\mathrm e}^{\alpha_{j} \sin x},
\end{equation}
where we need to find $c_{j} \in \mathbb{C} \setminus \{ 0 \}$  and pairwise different $\alpha_{j} \in {\mathbb{C}}$. 
Here, ${\mathrm e}^{\alpha_{j} \sin x}$ is of the form $H(x) {\mathrm e}^{\alpha_{j} G(x)}$ with $H(x) \coloneqq1$ and $G(x)\coloneqq \sin(x)$. To ensure that $G(x)$ is strictly monotone, we choose the interval $(-\frac{\pi}{2} \, \frac{\pi}{2})$.
With $g(x) = (G'(x))^{-1} = (\cos (x))^{-1}$ and $h(x) = 0$. We define the differential operator 
 $Af(x) =  (\cos (x))^{-1} f'(x)$ and find
$$ A  ({\mathrm e}^{\alpha_{j} \sin(\cdot)}) (x) = \frac{1}{\cos(x)} (\alpha_{j} \cos(x) \, {\mathrm e}^{\alpha_{j} \sin(x)}) = \alpha_{j} \, {\mathrm e}^{\alpha_{j} \sin(x)}.$$
According to Theorem \ref{theo1} we can therefore reconstruct $f$ in (\ref{esin}) from the derivative samples $f^{(\ell)}(x_0) $ for some $x_0 \in (-\frac{\pi}{2}, \frac{\pi}{2})$.

Using Theorem \ref{theo2}, we define with $H(x) \coloneqq1$ and $G(x)\coloneqq \sin(x)$ the generalized shift operator 
$$ S_{H,G,h} f (x) = f(G^{-1}(h+ G(x)) = f(\arcsin(h+\sin(x))).$$
We have to choose $x_0$ and $h$ such that all samples $f(\arcsin(h\ell+\sin(x_0)))$ are well-defined, i.e., $\sin(x_0) + h\ell \in (-\frac{\pi}{2}, \frac{\pi}{2})$ for $\ell=0, \ldots, 2M-1$. This is ensured for $a=-\frac{\pi}{2}+ \frac{h}{2}$ and $0 < h \le 1/(M+1)$. 

\begin{example}
We illustrate the reconstruction of a function $f(x)$ of the form (\ref{esin}) with $M=10$ and 
with real parameters $c_{j}$ and $\alpha_{j}$ in Table 2. The have been obtained by applying a uniform random sampling from the intervals $(-3,3)$ for $c_{j}$ and from $(-\pi, \pi)$ for $\alpha_{j}$. We chose a sampling distance $h=\frac{1}{17}$ and a starting point $x_0=-\frac{\pi}{2}+\frac{h}{2}=-\frac{\pi}{2}+\frac{1}{34}$.

\begin{table}[h]
\centering
\begin{tabular}{c|r|r|r|r|r|r|r|r|r|r}
 & $j=1$ & $j=2$& $j=3$& $j=4$ & $j=5$ &$j=6$ &$j=7$&$j=8$ &$j=9$ &$j=10$ \\[1ex] \hline
 $ c_j$ & $ ~2.104$ & $~0.363 $ & $~2.578$ & $~1.180$ & $~0.497$  & $~1.892$ & $~2.274$ & $~2.933$ &$-2.997 $& $~2.192$ \\[1ex]
 $\alpha_j$  &  $1.499$ & $0.540$ & $-1.591$ & $1.046$ & $-2.619$ & $0.791$ & $1.011$ & $1.444$ & $2.455$ & $3.030$\\[1ex]
\end{tabular}
 \medskip
 \caption{Parameters $c_{j}$ and $\alpha_{j}$ for $f(x)$ in (\ref{esin}) with $M=10$, see Figure \ref{fig:esin}.}
\end{table}

The reconstruction problem is ill-posed, and we cannot reconstruct the exact parameters with high precision, however, the reconstructed function is  a very good approximation of $f$.

\begin{figure}
  \centering
      \includegraphics[width=0.8\textwidth, height=60mm]{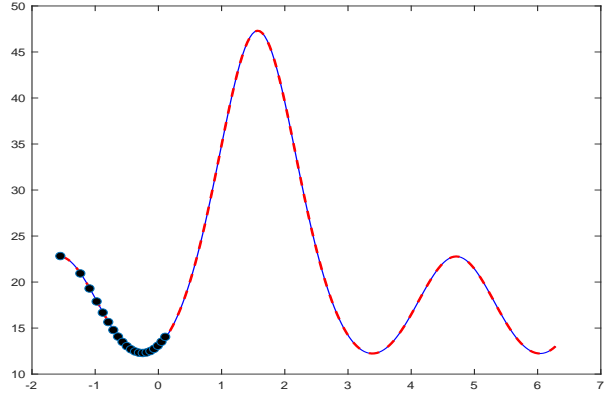} 
      \caption{Signal $f(x)$ in (\ref{esin}) consisting of $M=10$ terms according to Table 2.  The black dots indicate the used signal values and the reconstructed signal is shown in red.}
  \label{fig:esin}
\end{figure}
\end{example}

\section{Numerical Treatment of the Generalized Prony Method}\label{sec4}

In this section, we consider some numerical procedures to recover the parameters $\alpha_{j}, \, c_{j}$, $j=1, \ldots , M$, in (\ref{expGH}) resp. (\ref{f2}). 

\subsection{The simple Prony Algorithm}
First we summarize the direct algorithm for the recovery of $f$ in (\ref{f2}) from the function values $f(G^{-1}(h\ell + G(x_0)))$, $\ell=0, \ldots , 2M-1$,  according to the proof of Theorem \ref{theo2}.
\medskip

\noindent
\textbf{Algorithm 1.} \\[1ex]
\textbf{Input:} $M \in \mathbb{N}$, $h > 0$, sampled values $f(G^{-1}(h \ell + G(x_0)))$, $\ell=0, \ldots , 2M-1$.
\begin{enumerate}
\item Solve the linear system (\ref{lini}) to find the vector ${\mathbf p} = (p_{0}, \ldots , p_{M-1})^{T}$.
\item Compute all zeros  $z_{j} \in {\mathbb{C}}$, $j=1, \ldots , M$, of $p(z) = \sum\limits_{k=0}^{M-1} p_{k} \, z^{k} + z^{M}$.
\item Extract the coefficients $\alpha_{j}\coloneqq \frac{1}{h} \log z_{j}$ from $z_{j}= {\mathrm e}^{\alpha_{j}h}$, $j=1, \ldots. M$.
\item Solve the system (\ref{cj}) to compute $c_{1}, \ldots , c_{M} \in {\mathbb{C}}$.
\end{enumerate}
\textbf{Output:} $\alpha_{j} \in {\mathbb{R}} + {\mathrm i} [-\frac{\pi}{h}, \,\frac{\pi}{h} )$, $c_{j} \in {\mathbb{C}}$, $j=1, \ldots , M$.

\medskip

The assumptions of Theorem \ref{theo2} imply that the coefficient matrix of the linear system (\ref{lini}) is the invertible Hankel matrix,
$$ {\mathbf H}_{M} \coloneqq \left( \frac{f(G^{-1}(h(k+m) + G(a)))}{H(G^{-1}(h(k+m) + G(a)))} \right)_{k,m=0}^{M-1}. $$
The factorization (\ref{fak}) indicates that ${\mathbf H}_{M}$ may have very high condition that particularly depends on the condition of the Vandermonde matrix $\left( \mathrm{e}^{\alpha_{j}hm} \right)_{m=0,j=1}^{M-1,M}$. 

\subsection{ESPRIT for the Generalized Prony Method}\label{subsecesp}

We are interested in a more stable implementation of the recovery method  and present a modification of the ESPRIT method, see \cite{RK89,PT13,PT14} for the classical exponential sum.
We assume that  the number of terms $M$ in (\ref{expGH}) is not given beforehand, but $L$ is a known upper bound of $M$. 
In the following, we use the  notation ${\mathbf A}_{K,N}$ for  a rectangular matrix in ${\mathbb{C}}^{K\times N}$ and ${\mathbf A}_{K}$ for a square matrix in ${\mathbb{C}}^{K \times K}$, i.e., the subscripts indicate the matrix dimension.

Let 
\begin{equation}\label{hl} f_{\ell} \coloneqq \frac{f(G^{-1}(h \ell + G(x_0)))}{H(G^{-1}(h \ell+ G(x_0)))} , \qquad \ell=0, \ldots , 2N-1, 
\end{equation}
 be given and well defined, where $N \ge L \ge M$.
 
We consider first the rectangular Hankel matrix 
 $$ {\mathbf H}_{2N-L, L+1} \coloneqq \left( f_{\ell+m} \right)_{\ell,m=0}^{2N-L-1,L} \in {\mathbb{C}}^{(2N-L) \times(L+1)}.
$$ For exact data, (\ref{p0}) implies that rank ${\mathbf H}_{2N-L,L+1} = M$.
We therefore compute the singular value decomposition  of ${\mathbf H}_{2N-L,L+1}$,
\begin{equation}\label{sing}
{\mathbf H}_{2N-L,L+1} = {\mathbf U}_{2N-L} \, {\mathbf D}_{2N-L,L+1} \, {\mathbf W}_{L+1},
\end{equation}
with unitary square matrices ${\mathbf U}_{2N-L}$, ${\mathbf W}_{L+1}$ and a rectangular diagonal matrix ${\mathbf D}_{2N-L,L+1}$ containing the singular values of ${\mathbf H}_{2N-L,L+1}$.
We now determine the numerical rank $M$ of ${\mathbf H}_{2N-L,L+1}$ by inspecting  its singular values $\widetilde\sigma_{1} \ge \widetilde\sigma_{2} \ge  \ldots \ge \widetilde\sigma_{L+1} \ge 0$. 
We find $M$ as the number of singular values being larger than a predefined bound $\epsilon$. Usually, we can find a gap between $\widetilde{\sigma}_{M}$ and the further singular values $\widetilde\sigma_{M+1}, \ldots , \widetilde\sigma_{L+1}$, which are close to zero.
We now redefine the Hankel matrix and consider 
$ {\mathbf H}_{2N-M, M+1} \coloneqq \left( f_{\ell+m} \right)_{\ell,m=0}^{2N-M-1,M} \in {\mathbb{C}}^{(2N-M) \times(M+1)} $ with the corresponding SVD
\begin{equation} \label{S}{\mathbf H}_{2N-M,M+1} = {\mathbf U}_{2N-M} \, {\mathbf D}_{2N-M,M+1} \, {\mathbf W}_{M+1}, 
\end{equation}
with unitary matrices ${\mathbf U}_{2N-M}$ and ${\mathbf W}_{M+1}$. For exact data, ${\mathbf H}_{2N-M, M+1}$ has rank $M$, and
${\mathbf D}_{2N-M,M+1}^{T}= \left( \textrm{diag} (\sigma_{1}, \ldots , \sigma_{M}, \, 0), {\mathbf 0} \right) \in {\mathbb{R}}^{(M+1) \times (2N-M)}$ with $\sigma_{1}\ge \sigma_{2} \ge \ldots \ge \sigma_{M} > 0$.

We introduce  the sub-matrices ${\mathbf H}_{2N-M,M}(0)$ and ${\mathbf H}_{2N-M,M}(1)$ given by 
$$ \widetilde{\mathbf H}_{2N-M,M+1} \!= \! \left( {\mathbf H}_{2N-M,M}(0) ,  (f_{\ell+M})_{\ell=0}^{2N-M-1} \right) \!=\! \left( (f_{\ell})_{\ell=0}^{2N-M-1}, {\mathbf H}_{2N-M,M}(1) \right), $$
i.e., we obtain ${\mathbf H}_{2N-M,M}(0)$ be removing the last column of ${\mathbf H}_{2N-M,M+1}$ and ${\mathbf H}_{2N-M,M}(1)$ by removing the first column of ${\mathbf H}_{2N-M,M+1}$.
For exact data (\ref{lini}) yields   
\begin{equation}\label{H}
{\mathbf H}_{2N-M,M}(0) \, {\mathbf p}  = - \left( f_{\ell+M} \right)_{\ell=0}^{2N-M-1}, 
\end{equation}
where ${\mathbf p} = (p_{0}, \ldots , p_{M-1})^{T}$  contains the coefficients of the Prony polynomial in (\ref{pro}). Let 
$$ {\mathbf C}_{M}({\mathbf p}) \coloneqq \left( \begin{array}{ccccr}
0 \, & 0 \, & \ldots & \, 0  & -p_{0} \\
1 \, & 0 \,  & \ldots & \, 0 & -p_{1} \\
0 \, & 1 \, & \ldots & \, 0 & -p_{2} \\
\vdots & \, \vdots &  & \vdots & \vdots \\
0 \, & 0 \, & \ldots & \, 1 & - p_{M-1} \end{array} \right) \in {\mathbb{C}}^{M \times M}$$
be the (unknown) companion matrix of ${\mathbf p}$ possessing the $M$ zeros of $p(z) $ in (\ref{pro}) as eigenvalues. By (\ref{H}) it follows that 
\begin{equation}\label{C}
{\mathbf H}_{2N-M,M}(0) \, {\mathbf C}_{M}({\mathbf p}) = {\mathbf H}_{2N-M,M}(1). 
\end{equation} 
This observation leads to the following algorithm. 
According to (\ref{S})  we find the factorizations
\begin{eqnarray*}
{\mathbf H}_{2N-M,M}(0) &=&  {\mathbf U}_{2N-M} \, {\mathbf D}_{2N-M,M+1} \, {\mathbf W}_{M+1,M}(0), \\ 
{\mathbf H}_{2N-M,M}(1) &=&  {\mathbf U}_{2N-M} \, {\mathbf D}_{2N-M,M+1} \, {\mathbf W}_{M+1,M}(1) ,
\end{eqnarray*}
where ${\mathbf W}_{M+1,M}(0)$ is obtained by removing the last column of ${\mathbf W}_{M+1}$ and $ {\mathbf W}_{M+1,M}(1)$ by removing its first column. Now, (\ref{C}) implies 
$$ {\mathbf D}_{2N-M,M+1} {\mathbf W}_{M+1,M}(0) \, {\mathbf C}_{M} ({\mathbf p}) = {\mathbf D}_{2N-M,M+1} {\mathbf W}_{M+1,M}(1),$$
and by multiplication with the generalized inverse 
$${\mathbf D}_{2N-M,M+1}^{\dagger} = \left( \textrm{diag}\, (\frac{1}{\sigma_{1}}, \ldots \frac{1}{\sigma_{M}}, 0), {\mathbf 0} \right) \in {\mathbb{R}}^{(M+1) \times (2N-M)}.$$
Finally,
$$ {\mathbf W}_{M}(0) \, {\mathbf C}_{M} ({\mathbf p})  =  {\mathbf W}_{M}(1), $$
where the square matrices ${\mathbf W}_{M}(0)$ and ${\mathbf W}_{M}(1)$ are obtained from  ${\mathbf W}_{M+1,M}(0)$ and ${\mathbf W}_{M+1,M}(1)$, respectively, by removing the last row. Thus, the eigenvalues of ${\mathbf C}_{M} ({\mathbf p})$ are equal to the eigenvalues of 
$$ {\mathbf W}_{M}(0)^{-1} \, {\mathbf W}_{M}(1), $$
where ${\mathbf W}_{M}(0)$ is invertible since ${\mathbf C}_{M} ({\mathbf p})$ is invertible. (We can assume here the $z_{j} \neq 0$ since $z_{j} = {\mathrm e}^{\alpha_{j}}$.)
We therefore obtain the following new algorithm.
\medskip

\noindent
\textbf{Algorithm 2} (ESPRIT for the generalized Prony method)\\
\textbf{Input:} $L,N \in \mathbb{N}$, $L\le N$, $L$ upper bound for the number  $M$ of  terms in (\ref{f2}), sample values $f_{\ell}$, $\ell=0, \ldots , 2N-1$ as given in (\ref{hl}), $G(x_0)$.
\begin{enumerate}
\item Compute the SVD of the rectangular Hankel matrix ${\mathbf H}_{2N-L,L}$ as in (\ref{sing}). Determine the numerical rank $M$ of ${\mathbf H}_{2N-L,L}$, and compute the SVD of ${\mathbf H}_{2N-M,M+1}= {\mathbf U}_{2N-M} \, {\mathbf D}_{2N-M,M+1} \, {\mathbf W}_{M+1}$.
\item Build the restricted matrix  ${\mathbf W}_{M}(0)$  by removing the last column  and the last row  of ${\mathbf W}_{M+1}$ and ${\mathbf W}_{M}(1)$ by removing the  first column  and the last row of ${\mathbf W}_{M+1}$.
Compute the eigenvalues $z_{j}$ $j=0, \ldots , M$ of ${\mathbf W}_{M}(0)^{-1} {\mathbf W}_{M}(1)$.
\item Extract the coefficients $\alpha_{j}\coloneqq \frac{1}{h} \log z_{j}$ from $z_{j}= {\mathrm e}^{\alpha_{j}h}$, $j=1, \ldots. M$.
\item Solve the overdetermined system 
$$  f(\ell ) = \sum_{j=1}^{M} c_{j} \, z_{j}^{G(x_0)/h} \, z_{j}^{\ell}, \qquad \ell=0, \ldots , 2N-1,$$
to compute $c_{1}, \ldots , c_{M} \in {\mathbb{C}}$.
\end{enumerate}
\textbf{Output:} $M$, $\alpha_{j} \in {\mathbb{R}} + {\mathrm i} [-\frac{\pi}{h}, \,\frac{\pi}{h} )$, $c_{j} \in {\mathbb{C}}$, $j=1, \ldots , M$.
\medskip

\begin{example}
We compare the performance of the classical Prony method in Algorithm 1 with the ESPRIT method in Algorithm 2 and focus on the reconstruction of the frequency parameters. In our numerical example we choose $M=5$ and the parameter vectors $\alphar= (\alpha_{j})_{j=1}^{M}$, ${\mathbf c} = (c_{j})_{j=1}^{M} $ as 
$$
\alphar=(\frac{\pi}{2}, \frac{\mathrm{i}\pi}{4}, 0.4+\mathrm{i},-0.5,-1)^T \textrm{ and } {\mathbf c}=(0.5,2,-3,0.4i,-0.2)^T. 
$$
For the ESPRIT Algorithm 2 we have used $N=15$, i.e. $30$ sample values, and have fixed an upper bound $L=10$. For the rank approximation we have applied a bound $\epsilon=10^{-8}$. In Table \ref{tab3}, we present the results of parameter reconstruction using Algorithms 1 and 2. 
\begin{table}\label{tab3}
\centering
\begin{tabular}{c|r|r|r}
 $j$ & exact $\alpha_j$ & $\alpha_{j} \quad$ (Algorithm 1)&$\alpha_{j} \quad $(Algorithm 2)\\[0.5ex] \hline
 $j=1$ & $\frac{\pi}{2}$& $ 1.57121 + 6.0886\cdot 10^{-5}\mathrm{i}$ &$1.57079 - 2.3198\cdot10^{-8}\mathrm{i}$ \\[0.5ex]
 $j=2$ & $ \frac{\mathrm{i} \pi}{4}$ &$0.00231 + 0.7928\mathrm{i} $ & $2.00492\cdot 10^{-6}+ 0.7854\mathrm{i}$ \\[0.5ex]
$j=3$ & $0.4+\mathrm{i}$ & $0.40168 + 0.9982\mathrm{i}$ & $0.4000 + 1\mathrm{i}$\\[0.5ex]
$j=4$ & $-0.5$&$-0.49944 - 0.0013\mathrm{i}$ & $-0.5 - 4.3008\cdot 10^{-07}\mathrm{i}
$\\[0.5ex]
 $j=5$ & $-1$& $-1.00019 - 0.0042\mathrm{i}$ & $-1.0 - 1.1763^{-06}\mathrm{i}$
\end{tabular}
\end{table}
\end{example}

\begin{remark} The Hankel matrices occurring in the considered reconstruction problems can have a very high  condition. However, there are stable algorithms available to compute the SVD for Hankel matrices, particularly for the square case, see e.g. \cite{Drm15}.
\end{remark}

\subsection{Simplification in the Case of partially Known Frequency Parameters}\label{subsecpart}

In some applications it may occur that one or more of the parameters $\alpha_{j}$, or equivalently $z_{j}= {\mathrm e}^{\alpha_{j}h}$, are already known beforehand. However, if the corresponding coefficients $c_{j}$ are unknown, we cannot just eliminate the term $c_{j} \, H(x) \, {\mathrm e}^{\alpha_{j} G(x)}$ from the sum in (\ref{f2}) to get new measurements of the simplified sum from the original measurements. However, we can use the following approach.
Recall that the vector $\widetilde{\mathbf p}=(p_{0}, \ldots , p_{M})^{T}$ of coefficients of the Prony polynomial 
$$ p(z) =  \sum_{k=0}^{M} p_{k} z^{k}  = \prod_{j=1}^{M} (z- z_{j}) $$
satisfies 
$$ {\mathbf H}_{2N-M,M+1} \, \widetilde{\mathbf p} = {\mathbf 0}, $$
with ${\mathbf H}_{2N-M,M+1}$as in (\ref{hl}). 
Assume that $z_{1}$ is already known beforehand, and let 
$$q(z) = \sum_{k=0}^{M-1} q_{k} z^{k} =   \prod_{j=2}^{M} (z- z_{j}),$$ with the coefficient vector  $(q_{0}, \ldots , q_{M-1})^{T}$. Then  
$p(z) = (z-z_{1}) q(z)$ implies for the coefficient vectors
$$ \widetilde{\mathbf p} = \left( \begin{array}{c} 0 \\ q_{0} \\ \vdots \\ q_{M-1} \end{array} \right) - z_{1} \left( \begin{array}{c} q_{0} \\ \vdots \\ q_{M-1} \\ 0 \end{array} \right). $$
Thus
$$ {\mathbf H}_{2N-M,M+1} \widetilde{\mathbf p} = \left( {\mathbf H}_{2N-M,M}(1)- z_{1} {\mathbf H}_{2N-M,M}(0)\right) \, {\mathbf q} = {\mathbf 0},
$$
with ${\mathbf H}_{2N-M,M}(0)$  and ${\mathbf H}_{2N-M,M}(1)$ denoting the submatrices of ${\mathbf H}_{2N-M,M+1}$ where either the last column or the first column is removed.
Therefore, we easily find the new Hankel matrix 
$$ \widetilde{\mathbf H}_{2N-M,M} =  {\mathbf H}_{2N-M,M}(1)- z_{1} {\mathbf H}_{2N-M,M}(0) $$
 for the reduced problem. 
Observe from (\ref{hl}), that the new values in ${\mathbf H}_{2N-M,M}(1) - z_{1} {\mathbf H}_{2N-M,M}(0)$ are of the form 
\begin{eqnarray*}
\widetilde{f}_{\ell} =   f_{\ell+1} - z_{1} f_{\ell} &=&  
\sum_{j=1}^{M} c_{j} {\mathrm e}^{\alpha_{j}(f(\ell+1) + G(x_0))} - {\mathrm e}^{\alpha_{1}h} \sum_{j=1}^{M} c_{j} {\mathrm e}^{\alpha_{j}(f\ell + G(x_0))}  \\
&=& \sum_{j=2}^{M} c_{j} ({\mathrm e}^{\alpha_{j}h}-{\mathrm e}^{\alpha_{1}h}) {\mathrm e}^{\alpha_{j}(h\ell + G(x_0))}, 
\end{eqnarray*}
i.e., the coefficients $c_{j}$,$j=2,\dots,M$, are changed to $\widetilde{c}_{j}=c_{j} ({\mathrm e}^{\alpha_{j}h}-{\mathrm e}^{\alpha_{1}h})$. 
Thus, we can use the samples $\widetilde{f}_{\ell}$ to recover the shorter sum $\sum\limits_{j=2}^{M} \widetilde{c}_{j}  H(x) {\mathrm e}^{\alpha_{j}G(x)}$. Once we have computed the remaining $\alpha_j$,$j=2,\dots,M$ we compute the coefficients $c_j$ by solving system (\ref{cj}) for $j=1,\dots,M$.  

\subsection{Recovery of Exponential Sums from Non-equispaced Data}\label{subsecnoneq}
This section is devoted to the recovery of a signal $f$ of the form
\begin{align*}
f(x)=\sum_{j=1}^M c_j e^{\alpha_j x} 
\end{align*}
with $c_j \in \mathbb{C} \setminus \{0\}$, $\alpha_j \in \mathbb{C}$ for $j=1,\dots,M$ using non-equidistant data 
$$ a = y_{0} < y_{1} < \ldots  < y_{2N-1} = b. 
$$
Similarly to the ESPRIT-like algorithm in Section \ref{subsecesp} $M\le L \le N$ holds. In order to recover $f$ from the non-equispaced samples we try to find a continuous strictly monotone function $G^{-1}$ on $[a,b]$ such that 
\begin{align*}
G^{-1}(y_{\ell})=\ell h ~~~~\ell=0,\dots,2N-1. 
\end{align*}
 We employ the substitution
$y\coloneqq G^{-1}(x)$ with $ x\in [a,b]$.
\begin{align*}
f(x)=\sum_{j=1}^M c_j e^{\alpha_j G\left(G^{-1}(x) \right) }= \sum_{j=1}^M c_j e^{\alpha_j G(y)}.
\end{align*} 
This is a special form of the model (\ref{expGa}) with $H(x)\equiv 1$ and can be recovered using the ESPRIT-like Algorithm. 
%
In some applications, the nodes $y_{\ell}$  may already satisfy a predefined known structure that provides us with the function $G^{-1}$. Otherwise, we can find some $G^{-1}$ by solving the interpolation problem above.
Theorem 2 already holds for continuous and strictly monotone functions $G$. Therefore, one simple choice for $G^{-1}$ would be  a piecewise linear spline  function with 
\begin{equation}\label{G}
G^{-1}(y) = \ell \, h + \left( \frac{y-y_{\ell}}{y_{\ell+1}- y_{\ell}} \right) \, h \quad \textrm{for} \; y \in [y_{\ell}, \, y_{\ell+1}], \; \ell=0, \ldots , 2N-2, 
\end{equation}
which is strictly increasing. We summarize the algorithm  for the recovery of the exponential sum with Im $\alpha_{j}  \in (-T, T]$.

\medskip

\noindent
\textbf{Algorithm 3} (ESPRIT for non-equispaced sampled data)\\
\textbf{Input:} $L,N \in \mathbb{N}$, $L\le N$, $L$ upper bound for the number  $M$ of  terms in (\ref{exp}), 
$h$ with $0 < h \le \frac{\pi}{T}$, sample values $f(y_{\ell})$, $\ell=0, \ldots , 2N-1$, where $y_{0} < y_{1} < \ldots < y_{2N-1}$.
\begin{enumerate}
\item Compute  a continuous strictly monotone  function $G^{-1}$ satisfying the interpolation conditions 
$$ G^{-1}(y_{\ell}) = h \ell, \qquad \ell=0, \ldots , 2N-1, $$
as e.g. in (\ref{G}). 
\item Apply Algorithm 2  with $G^{-1}(y)$ as determined in step 1 and with $H(y) \equiv 1$ using the samples $f(y_{\ell}) = f(G(h\ell))$.
\end{enumerate}
\textbf{Output:} $M$, $\alpha_{j} \in {\mathbb{R}} + {\mathrm i} [\frac{-\pi}{h}, \,\frac{\pi}{h} )$, $c_{j} \in {\mathbb{C}}$, $j=1, \ldots , M$.
\medskip

\begin{remark}
In \cite{PT14}, another procedure for recovering the exponential sum from non-equispaced samples has been proposed, which is essentially based on the idea that the given discrete samples are first interpolated by a spline function, and then the equidistant samples of the obtained spline are applied in the usual ESPRIT algorithm. 
\end{remark}

\section{Modified Prony Method for Sparse Approximation}\label{sec5}

In this section, we want to consider the question, how to approximate a given data vector ${\mathbf y} =(y_{k})_{k=0}^{L}$ with $L \ge 2M-1$ by a new  vector $\mathbf{f}=(f_{k})_{k=0}^{L}$ whose elements are structured as 
$$ f_{k} = \sum_{j=1}^{M} c_{j} \,  z_{j}^{k},$$
i.e., $\mathbf{f}$ only depends on the parameter vectors $\mathbf{c} = (c_{j})_{j=1}^{M}$ and ${\mathbf z} = (z_{j})_{j=1}^{M}$. We assume that 
for the given data ${\mathbf y}$ the corresponding Hankel matrix ${\mathbf H} \coloneqq (y_{k+m})_{k=0,m=0}^{L-M-1,M-1}$ has full rank, i.e., that the given data  cannot be exactly represented  by an exponential sum  with less  than $M$ terms, as it can be also seen from the factorization (\ref{fak}).
Therefore, we can  suppose that $z_{j} \in {\mathbb{C}}$ are pairwise distinct and 
$c_j \in {\mathbb{C}} \setminus \{ 0 \}$. 

\subsection{The nonlinear least-squares problem}\label{subsecleast}
We want to solve the minimization problem 
\begin{equation}\label{min}
\mathop{\textrm{argmin}}\limits_{\mathbf{c},\mathbf{z} \in \mathbb{C}^{M}} \left\| {\mathbf y} - \left( \sum_{j=1}^{M} c_{j} \, z_{j}^{k} \right)_{k=0}^{L} \right\|_{2}.
\end{equation}
This problem occurs in two different scenarios. The first one is the problem of parameter estimation in case of noisy data. Assume that we have noisy samples $y_{k}=f(k) + \epsilon_{k}$, $k=0, \ldots , L$, where 
$\epsilon_{k}$ are i.i.d.\ random variables with $\epsilon_{k} \in N(0, \sigma^{2})$. 
In the second scenario we consider the sparse nonlinear approximation problem to find a function $f(x) = \sum_{j=1}^{M} c_{j} z_{j}^{x}$, which minimizes $\sum_{\ell=0}^{L} |y_{\ell} - f(\ell) |^{2}$. 
With the Vandermonde matrix 
$$ {\mathbf V}_{\mathbf z} \coloneqq \left( \begin{array}{cccc}
1 & 1 & \ldots & 1 \\
z_{1} & z_{2} & \ldots & z_{M} \\
z_{1}^{2} & z_{2}^{2} &\ldots & z_{M}^{2} \\
\vdots & \vdots & & \vdots \\
z_{1}^{L} & z_{2}^{L} &\ldots & z_{M}^{L} \end{array} \right) \in \mathbb{C}^{(L+1) \times M} $$
we have ${\mathbf f}= {\mathbf V}_{\mathbf z}\, {\mathbf c}$, and  the problem (\ref{min}) can be reformulated as
$$ \mathop{\textrm{argmin}}\limits_{\mathbf{c},\mathbf{z} \in \mathbb{C}^{M}} \| {\mathbf y} - {\mathbf V}_{\mathbf z} {\mathbf c}\|_{2}. $$
For given ${\mathbf z}$, the linear least squares problem 
$ \mathop{\textrm{argmin}}\limits_{\mathbf{c} \in \mathbb{C}^{M}}  \| {\mathbf y} - {\mathbf V}_{\mathbf z} {\mathbf c} \|_{2}$ 
can be directly solved, and we obtain ${\mathbf c} = {\mathbf V}_{\mathbf z}^{+} {\mathbf y} = [{\mathbf V}_{\mathbf z}^{*} {\mathbf V}_{\mathbf z}]^{-1} {\mathbf V}_{\mathbf z}^{*} {\mathbf y}$, since ${\mathbf V}_{\mathbf z}$ has  full rank $M$.
Thus (\ref{min}) can be simplified to
\begin{eqnarray*}
 \mathop{\textrm{argmin}}\limits_{\mathbf{z} \in \mathbb{C}^{M}}  \| {\mathbf y} - {\mathbf V}_{\mathbf z} {\mathbf V}_{\mathbf z}^{+} {\mathbf y} \|_{2}^{2}
&=& \mathop{\textrm{argmin}}\limits_{\mathbf{z} \in \mathbb{C}^{M}} \| ({\mathbf I}-{\mathbf P}_{\mathbf z} ) {\mathbf y} \|_{2}^{2} \\
&=& \mathop{\textrm{argmin}}\limits_{\mathbf{z} \in \mathbb{C}^{M}} ({\mathbf y}^{*}{\mathbf y} - {\mathbf y}^{*} {\mathbf P}_{\mathbf z} {\mathbf y})
= \mathop{\textrm{argmax}}\limits_{\mathbf{z} \in \mathbb{C}^{M}} {\mathbf y}^{*} {\mathbf P}_{\mathbf z} {\mathbf y}, 
\end{eqnarray*}
where ${\mathbf P}_{\mathbf z} = {\mathbf V}_{\mathbf z} {\mathbf V}_{\mathbf z}^{+}$ denotes the projection matrix satisfying ${\mathbf P}_{\mathbf z} = {\mathbf P}_{\mathbf z}^{*} = {\mathbf P}_{\mathbf z}^{2}$, ${\mathbf P}_{\mathbf z} {\mathbf V}_{\mathbf z} = {\mathbf V}_{\mathbf z}$ as well as ${\mathbf V}_{\mathbf z}^{+} {\mathbf P}_{\mathbf z} = {\mathbf V}_{\mathbf z}^{+}$.
Hence, similarly as for Prony's method, we can concentrate on finding the parameters  $z_{j}$ in ${\mathbf z}$ first.

Consider ${\mathbf r}({\mathbf z}) \coloneqq {\mathbf P}_{\mathbf z} {\mathbf y} \in {\mathbb{C}}^{L+1}$. Then the optimization problem is equivalent to 
\begin{equation}\label{opt}
\mathop{\textrm{argmax}}\limits_{\mathbf{z} \in \mathbb{C}^{M}} \| {\mathbf r}({\mathbf z}) \|_{2}^{2} = \mathop{\textrm{argmax}}\limits_{\mathbf{z} \in \mathbb{C}^{M}} \| {\mathbf P}_{\mathbf z} {\mathbf y} \|_{2}^{2}.
\end{equation}
To derive an iterative algorithm  for solving (\ref{opt}), we first determine the Jacobian ${\mathbf J}_{\mathbf z}$ of ${\mathbf r}({\mathbf z}) = \left( r_{\ell}({\mathbf z}) \right)_{\ell=0}^{L}$.

\begin{theorem}
The Jacobian matrix ${\mathbf J}_{\mathbf z} \in {\mathbb{C}}^{(L+1) \times M}$ of  ${\mathbf r}({\mathbf z}) $ in $(\ref{opt})$ is given by 
\begin{eqnarray} \nonumber
{\mathbf J}_{\mathbf z} &\coloneqq& \left( \frac{\partial  r_{\ell}({\mathbf z})}{\partial z_{j}} \right)_{\ell=0,j=1}^{L,M} \\
\label{J}
&=& ({\mathbf I}_{L+1} - {\mathbf P}_{\mathbf z}) {\mathbf V}_{\mathbf z}' \, \rm{diag} ({\mathbf V}_{\mathbf z}^{+} {\mathbf y}) +
({\mathbf V}_{\mathbf z}^{+})^{*} \, \rm{diag} (({\mathbf V}_{\mathbf z}')^{*} ({\mathbf I}_{L+1} - {\mathbf P}_{z}) {\mathbf y}),
\end{eqnarray}
where ${\mathbf I}_{L+1}$ denotes the identity matrix of size $L+1$, 
$$ {\mathbf V}_{\mathbf z}' \coloneqq \left( \begin{array}{cccc} 
0 & 0 & \ldots & 0 \\
1 & 1 & \ldots & 1 \\
2z_{1} & 2 z_{2} & \ldots &2 z_{M} \\
\vdots & \vdots & & \vdots \\
L z_{1}^{L-1} & L z_{2}^{L-1} & \ldots & L z_{M}^{L-1} \end{array} \right) \in \mathbb{C}^{(L+1) \times M}\quad \rm{and} \quad   \rm{diag}  ({\mathbf q}) \coloneqq \left( \begin{array}{cccc} 
q_{1} & 0  & \ldots & 0 \\
0 & q_{2} &  & 0 \\
\vdots &  & \ddots & \vdots \\
0 & 0 & \ldots & q_{M} \end{array} \right)
$$
for $ {\mathbf q}  \in {\mathbb{C}}^{M}$.  In particular, the gradient of $ \| {\mathbf r}({\mathbf z}) \|_{2}^{2}$  reads
\begin{equation}\label{grad} \nabla \| {\mathbf r}({\mathbf z}) \|_{2}^{2} = 2 {\mathbf J}_{\mathbf z}^{*} \,  {\mathbf r}({\mathbf z}) = 
\rm{diag} (({\mathbf V}_{\mathbf z}')^{T} ({\mathbf I}_{L+1} - {\mathbf P}_{z}) \overline{\mathbf y}) \, {\mathbf V}_{\mathbf z}^{+} {\mathbf y}.
\end{equation}
\end{theorem}

\begin{proof}
First, observe that $\frac{\partial }{\partial z_{j}} {\mathbf V}_{{\mathbf z}}$ is a rank-1 matrix of the form
$$ \frac{\partial }{\partial z_{j}} {\mathbf V}_{{\mathbf z}} = {\mathbf z}_{j}' \, {\mathbf e}_{j}^{*} \in {\mathbb{C}}^{(L+1) \times M}, \qquad j =1, \ldots , M,$$
where ${\mathbf z}_{j}' = (0, 1, 2z_{j}, 3z_{j}^{2}, \ldots , Lz_{j}^{L-1})^{T}$ and ${\mathbf e}_{j}$ is the $j$th unit vector of length $M$.
Then we obtain 
\begin{eqnarray*}
& & 
\frac{\partial }{\partial z_{j}} {\mathbf r}({\mathbf z}) = \frac{\partial }{\partial z_{j}}  \left( {\mathbf V}_{\mathbf z} [{\mathbf V}_{\mathbf z}^{*} {\mathbf V}_{\mathbf z}]^{-1} {\mathbf V}_{\mathbf z}^{*} {\mathbf y} \right) \\
&=& 
({\mathbf z}_{j}' \, {\mathbf e}_{j}^{*}) {\mathbf V}_{\mathbf z}^{+} {\mathbf y} - ({\mathbf V}_{\mathbf z}^{+})^{*} \left[ ({\mathbf z}_{j}' \, {\mathbf e}_{j}^{*})^{*} {\mathbf V}_{\mathbf z} + {\mathbf V}_{\mathbf z}^{*} ({\mathbf z}_{j}' \, {\mathbf e}_{j}^{*}) \right] {\mathbf V}_{\mathbf z}^{+} {\mathbf y} 
+ ({\mathbf V}_{\mathbf z}^{+})^{*} ({\mathbf z}_{j}' \, {\mathbf e}_{j}^{*})^{*} {\mathbf y} \\
&=&
({\mathbf V}_{\mathbf z}^{+} {\mathbf y})_{j}\, {\mathbf z}_{j}' -  (({\mathbf z}_{j}')^{*} {\mathbf P}_{\mathbf z} {\mathbf y}) ({\mathbf V}_{\mathbf z}^{+})^{*}  {\mathbf e}_{j} - ({\mathbf V}_{\mathbf z}^{+} {\mathbf y})_{j} {\mathbf P}_{\mathbf z} {\mathbf z}_{j}' + (({\mathbf z}_{j}')^{*} {\mathbf y}) ({\mathbf V}_{\mathbf z}^{+})^{*} {\mathbf e}_{j} \\
&=& ({\mathbf V}_{\mathbf z}^{+} {\mathbf y})_{j} ({\mathbf I}_{L+1} - {\mathbf P}_{{\mathbf z}})  {\mathbf z}_{j}' + 
 (({\mathbf z}_{j}')^{*} ({\mathbf I}_{L+1} - {\mathbf P}_{\mathbf z}) {\mathbf y}) ({\mathbf V}_{\mathbf z}^{+})^{*} {\mathbf e}_{j} \\
&=&  ({\mathbf V}_{\mathbf z}^{+} {\mathbf y})_{j} ({\mathbf I}_{L+1} - {\mathbf P}_{{\mathbf z}}) {\mathbf V}_{{\mathbf z}}' {\mathbf e}_{j} + (({\mathbf z}_{j}')^{*} ({\mathbf I}_{L+1} - {\mathbf P}_{\mathbf z}) {\mathbf y}) ({\mathbf V}_{\mathbf z}^{+})^{*} {\mathbf e}_{j},
\end{eqnarray*}
where  $({\mathbf V}_{\mathbf z}^{+} {\mathbf y})_{j}$ denotes the $j$th component of ${\mathbf V}_{\mathbf z}^{+} {\mathbf y}$. From this observation, we immediately get ${\mathbf J}_{\mathbf z}$ in (\ref{J}).
This formula implies further
\begin{eqnarray*}
{\mathbf J}_{\mathbf z}^{*} {\mathbf r}({\mathbf z}) &=&
(\rm{diag} \overline{{\mathbf V}_{\mathbf z}^{+} {\mathbf y}}) ({\mathbf V}_{\mathbf z}')^{*} ({\mathbf I} - {\mathbf P}_{\mathbf z}) {\mathbf P}_{\mathbf z} {\mathbf y} + 
(\rm{diag} (({\mathbf V}_{\mathbf z}')^{*} ({\mathbf I}_{L+1} - {\mathbf P}_{z}) {\mathbf y}))^{*} {\mathbf V}_{\mathbf z}^{+} {\mathbf P}_{\mathbf z} {\mathbf y} \\
&=& \rm{diag} (({\mathbf V}_{\mathbf z}')^{T} ({\mathbf I}_{L+1} - {\mathbf P}_{z}) \overline{\mathbf y}) \, {\mathbf V}_{\mathbf z}^{+} {\mathbf y}.
\end{eqnarray*}
\hfill \qed
\end{proof}

\begin{corollary}\label{coro}
Let ${\mathbf y} \in {\mathbb{C}}^{L+1}$ be given and assume that $(y_{k+m})_{k=0,m=0}^{L-M+1,M-1}$ has full rank $M$. Then, a vector ${\mathbf z} \in {\mathbb{C}}^{M}$ solving $(\ref{opt})$ necessarily satisfies 
$$ ({\mathbf V}_{\mathbf z}')^{*} ({\mathbf I}_{L+1} - {\mathbf P}_{\mathbf z}) {\mathbf y} = {\mathbf 0}.
$$
\end{corollary}
\begin{proof}
The assertion follows from (\ref{grad}) using the information that ${\mathbf c} = {\mathbf V}_{\mathbf z}^{+} {\mathbf y}$ has no vanishing components.
\hfill \qed
\end{proof}

\begin{remark}
1. The necessary condition in Corollary \ref{coro} can be used to build an iterative  algorithm for updating the vector ${\mathbf z}$ where we start with ${\mathbf z}^{(0)}$ obtained from the ESPRIT algorithm 2. Then we search for ${\mathbf z}^{(j+1)}$ by solving 
 $$({\mathbf V}_{{\mathbf z}^{(j+1)}}')^{*} ({\mathbf I}_{L+1} - {\mathbf P}_{{\mathbf z}^{(j)}}) {\mathbf y} = {\mathbf 0},$$
i.e., by computing the zeros of the polynomial with coefficient vector 
$$ \rm{diag} (0,1,2, \ldots , L) \, ({\mathbf I}_{L+1} - {\mathbf P}_{\mathbf z}^{(j)}) {\mathbf y}$$ 
and taking the subset of $M$ zeros which is closest to the previous set  ${\mathbf z}^{(j)}$.\\
2. This approach is different from most ideas  to solve (\ref{min}) in the literature, see e.g.\ \cite{BM86,OS91,OS95} and the recent survey \cite{ZP19}.  In these papers, one first transfers  the problem  of finding ${\mathbf z} \in {\mathbb{C}}^{M}$  into the problem of  finding the vector ${\mathbf p} =(p_{k})_{k=0}^{M} \in {\mathbb{C}}^{M+1}$ with $\| {\mathbf p} \|_{2}= 1$,  such that $p(z_{j}) = \sum_{k=0}^{M} p_{k} z_{j}^{k} =0$ for all $j=1, \ldots , M$, thereby imitating  the idea of Prony's method.
Introducing the matrix
$$ {\mathbf X}_{\mathbf p}^{T} = \left( \begin{array}{cccccccc}
p_{0} & p_{1} & \ldots & & p_{M} & & &  \\
 & p_{0} & p_{1} & \ldots &  & p_{M} & & \\
  & & \ddots & & & & \ddots & \\
  & & & p_{0}& p_{1} & \ldots &  & p_{M} \end{array} \right) \in \mathbb{C}^{(L-M+1) \times (L+1)}$$
that satisfies ${\mathbf X}_{\mathbf p}^{T} {\mathbf V}_{\mathbf z} = {\mathbf 0}$, we obtain a projection matrix 
 $$ \overline{\mathbf P}_{\mathbf p} \coloneqq \overline{\mathbf X}_{\mathbf p} \overline{\mathbf X}_{\mathbf p}^{+} = \overline{\mathbf X}_{\mathbf p}  [ {\mathbf X}_{\mathbf p}^{T} \overline{\mathbf X}_{\mathbf p}]^{-1} {\mathbf X}_{\mathbf p}^{T} = ({\mathbf I}_{L+1} - {\mathbf P}_{\mathbf z}), $$
and (\ref{opt}) can be rephrased   as 
$$ \mathop{\textrm{argmin}}\limits_{\mathbf{p} \in \mathbb{C}^{M+1} \atop \| {\mathbf p} \|_{2} = 1} \| \overline{\mathbf P}_{\mathbf b} {\mathbf y} \|_{2}^{2} = \mathop{\textrm{argmin}}\limits_{\mathbf{p} \in \mathbb{C}^{M+1} \atop \| {\mathbf p} \|_{2} = 1} {\mathbf y}^{*} \overline{\mathbf X}_{\mathbf p}  [ {\mathbf X}_{\mathbf p}^{T} \overline{\mathbf X}_{\mathbf p}]^{-1} {\mathbf X}_{\mathbf p}^{T} {\mathbf y}. $$
\end{remark}

\subsection{Gau{\ss}-Newton and Levenberg-Marquardt iteration}\label{subsecleven}
Another approach than given in Remark 3.1 to solve the non-linear least squares problem (\ref{opt}) is the following.
We approximate ${\mathbf r}({\mathbf z} + \deltar)$ using its first order Taylor expansion ${\mathbf r}({\mathbf z}) + {\mathbf J}_{\mathbf z} \deltar$. Now, instead of minimizing  $\| {\mathbf r}({\mathbf z} + \deltar) \|_{2}^{2}$ we consider 
$$
 \mathop{\textrm{argmin}}\limits_{\deltas \in \mathbb{C}^{M}} \|{\mathbf r}({\mathbf z}) + {\mathbf J}_{\mathbf z} \deltar \|_{2}^{2} =   \mathop{\textrm{argmin}}\limits_{\deltas \in \mathbb{C}^{M}} ( \|{\mathbf r}({\mathbf z}) \|_{2}^{2} + ({\mathbf r}({\mathbf z}))^{*} {\mathbf J}_{\mathbf z} \deltar + \deltar^{*}  {\mathbf J}_{\mathbf z}^{*} {\mathbf r}({\mathbf z}) + \deltar^{*}  {\mathbf J}_{\mathbf z}^{*}{\mathbf J}_{\mathbf z} \deltar )
$$
 which yields
 $$ 2 \rm{Re} ( {\mathbf J}_{\mathbf z}^{*} {\mathbf r}({\mathbf z})) + 2 {\mathbf J}_{\mathbf z}^{*} {\mathbf J}_{\mathbf z} \deltar = {\mathbf 0}.
 $$
 Thus, starting with the vector ${\mathbf z}^{(0)}$ obtained from Algorithm 2, the $j$th step of the Gau{\ss}-Newton iteration is of the form
 $$ ({\mathbf J}_{{\mathbf z}^{(j)}}^{*} {\mathbf J}_{{\mathbf z}^{(j)}} )  \deltar^{(j)} = - \rm{Re} ( {\mathbf J}_{{\mathbf z}^{(j)}}^{*} {\mathbf r}({\mathbf z}^{(j)}))
 $$
with ${\mathbf z}^{(j+1)} = {\mathbf z}^{(j)} + \deltar^{(j)}$. Since
 $({\mathbf I}_{L+1} - {\mathbf P}_{{\mathbf z}^{(j)}}) {\mathbf y}$ may be already close to the zero vector, the matrix $({\mathbf J}_{{\mathbf z}^{(j)}}^{*} {\mathbf J}_{{\mathbf z}^{(j)}} )$  is usually ill-conditioned. Therefore, we regularize  by changing the matrix in each step  to 
 $({\mathbf J}_{{\mathbf z}^{(j)}}^{*} {\mathbf J}_{{\mathbf z}^{(j)}} ) + \lambda_{j} {\mathbf I}_{M}$ and obtain the Levenberg-Marquardt iteration
  $$ (({\mathbf J}_{{\mathbf z}^{(j)}}^{*} {\mathbf J}_{{\mathbf z}^{(j)}} ) + \lambda_{j} {\mathbf I}_{M}) \deltar^{(j)} = - \rm{Re} ( {\mathbf J}_{{\mathbf z}^{(j)}}^{*} {\mathbf r}({\mathbf z}^{(j)})).
 $$
 In this algorithm, we need to fix the parameters $\lambda_{j}$, which are usually chosen very small.
If we arrive at a (local) maximum, then the right-hand side in the Levenberg-Marquardt iteration vanishes, and we obtain $\deltar^{(j)} = {\mathbf 0}$.

\begin{remark}
1. The considered non-linear least squares problem is also closely related to  structured low-rank approximation, see \cite{Mar18,UM14}.
Further, instead of the Euclidean norm, one can consider the maximum norm, see \cite{BH05,Ha19} or the 1-norm, see \cite{Skr17}.

2. Some questions remain. How good is this approximation and what is the rate of convergence with respect to $M$.The authors are not aware of a complete answer to this question. However, in \cite{BH05} it has been shown that the function $1/x$ can be approximated by an exponential sum with an error ${\mathcal O}(\exp(c\sqrt{M})$. Also the results in \cite{BM05} and \cite{PP19} indicate that we can hope for an exponential decay of the approximation error for a larger class of functions.

\end{remark}
%
%
\bibliography{myBib}{}

\begin{thebibliography}{10}

\bibitem{AAK71}
Adamjan, V., Arov, D., Krein, M.:
\newblock Analytic properties of the {S}chmidt pairs of a {H}ankel operator and
  the generalized {S}chur-{T}akagi problem.
\newblock Mat. Sb. \textbf{86} (1971)  34--75

\bibitem{ACH11}
Andersson, F., Carlsson, M., de~Hoop, M.:
\newblock Sparse approximation of functions using sums of exponentials and
  {AAK} theory.
\newblock J. Approx. Theory \textbf{163} (2011)  213--248

\bibitem{BSBV17}
Baechler, G., Scholefield, A., Baboulaz, L., Vetterli, M.:
\newblock Sampling and exact reconstruction of pulses with variable width.
\newblock IEEE Trans. Signal Process. \textbf{65}(10) (2017)  2629--2644

\bibitem{Bar05}
Barone, P.:
\newblock On the distribution of poles of {P}adé approximants to the
  {Z}-transform of complex {G}aussian white noise.
\newblock J. Approx. Theory \textbf{132}(2) (2005)  224--240

\bibitem{BM05}
Beylkin, G., Monz\'{o}n, L.:
\newblock On approximation of functions by exponential sums.
\newblock Appl. Comput. Harmon. Anal. \textbf{19} (2005)  17--48

\bibitem{BH05}
Braess, D., Hackbusch, W.:
\newblock Approximation of $1/x$ by exponential sums in $[1, \infty)$.
\newblock IMA J. Numer. Anal. \textbf{25} (2005)  685--697

\bibitem{BM86}
Bresler, Y., Macovski, A.:
\newblock Exact maximum likelihood parameter estimation of superimposed
  exponential signals in noise.
\newblock IEEE Trans. Acoust., Speech, Signal Process. \textbf{34}(5) (1986)
  1081--1089

\bibitem{chunaev16}
Chunaev, P., Danchenko, V.:
\newblock Approximation by amplitude and frequency operators.
\newblock J. Approx. Theory \textbf{207} (2016)  1--31

\bibitem{Cuyt18}
Cuyt, A., Tsai, M.n., Verhoye, M., Lee, W.s.:
\newblock Faint and clustered components in exponential analysis.
\newblock Appl. Math. Comput. \textbf{327} (2018)  93--103

\bibitem{Dra07}
Dragotti, P., Vetterli, M., Blu, T.:
\newblock Sampling moments and reconstructing signals of finite rate of
  innovation: {S}hannon meets {S}trang--{F}ix.
\newblock IEEE Trans Signal Process. \textbf{55}(5) (2007)  1741--1757

\bibitem{Drm15}
Drma\v{c}, Z.:
\newblock {SVD} of {H}ankel matrices in {V}andermonde-{C}auchy product form.
\newblock Electron. Trans. Numer. Anal. \textbf{44} (2015)  593--623

\bibitem{Ha19}
Hackbusch, W.:
\newblock Computation of best $l^{\infty}$ exponential sums for $1/x$ by
  {R}emez' algorithm.
\newblock Comput. Vis. Sci. \textbf{20}(1-2) (2019)  1--11

\bibitem{Hau90}
Hauer, J., Demeure, C., Scharf, L.:
\newblock Initial results in {P}rony analysis of power system response signals.
\newblock IEEE Trans. Power Systems \textbf{5}(1) (1990)  80--89

\bibitem{HS90}
Hua, Y., Sarkar, T.:
\newblock On the total least squares linear prediction method for frequency
  estimation.
\newblock IEEE Trans. Acoust. Speech Signal Process. \textbf{38}(12) (1990)
  2186--2189

\bibitem{Lang00}
Lang, M.C.:
\newblock Least-squares design of {IIR} filters with prescribed magnitude and
  phase responses and a pole radius constraint.
\newblock IEEE Trans. Signal Process. \textbf{48}(11) (2000)  3109--3121

\bibitem{Mano05}
Manolakis, D., Ingle, V., Kogon, S.:
\newblock Statistical and Adaptive Signal Processing.
\newblock McGraw-Hill, Boston (2005)

\bibitem{Mar18}
Markovsky, I.:
\newblock Low-Rank Approximation: Algorithms, Implementation, Applications.
\newblock Springer, second edition edition (2018)

\bibitem{OS91}
Osborne, M., Smyth, G.:
\newblock A modified {P}rony algorithm for fitting functions defined by
  difference equations.
\newblock SIAM J. Sci. Stat. Comput. \textbf{12} (1991)  362--382

\bibitem{OS95}
Osborne, M., Smyth, G.:
\newblock A modified {P}rony algorithm for exponential function fitting.
\newblock SIAM J. Sci. Comput. \textbf{16}(1) (1995)  119--138

\bibitem{PP13}
Peter, T., Plonka, G.:
\newblock A generalized {P}rony method for reconstruction of sparse sums of
  eigenfunctions of linear operators.
\newblock Inverse Problems \textbf{29}(2) (2013)

\bibitem{PP16}
Plonka, G., Pototskaia, V.:
\newblock Application of the {AAK} theory for sparse approximation of
  exponential sums.
\newblock arXiv:1609.09603 (2016)

\bibitem{PP19}
Plonka, G., Pototskaia, V.:
\newblock Computation of adaptive {F}ourier series by sparse approximation of
  exponential sums.
\newblock J. Fourier Anal. Appl. \textbf{25}(4) (2019)  1580--1608

\bibitem{PSK19}
Plonka, G., Stampfer, K., Keller, I.:
\newblock Reconstruction of stationary and non-stationary signals by the
  generalized {P}rony method.
\newblock Anal. and Appl. \textbf{17}(2) (2019)  179--210

\bibitem{PT14}
Plonka, G., Tasche, M.:
\newblock Prony methods for recovery of structured functions.
\newblock GAMM Mitt. \textbf{37}(2) (2014)  239--258

\bibitem{Poh10}
Poh, K., Marziliano, P.:
\newblock Compressive sampling of {EEG} signals with finite rate of innovation.
\newblock EURASIP J. Adv. Signal Process. (2010)  183105

\bibitem{PT10}
Potts, D., Tasche, M.:
\newblock Parameter estimation for exponential sums by approximate {P}rony
  method.
\newblock Signal Process. \textbf{90}(5) (2010)  1631--1642

\bibitem{PT13}
Potts, D., Tasche, M.:
\newblock Parameter estimation for multivariate exponential sums.
\newblock Electron. Trans. Numer. Anal. (40) (2013)  204--224

\bibitem{RK89}
Roy, R., Kailath, T.:
\newblock Esprit estimation of signal parameters via rotational invariance
  techniques.
\newblock IEEE Trans. Acoust. Speech Signal Process. \textbf{37} (1989)
  984--995

\bibitem{Skr17}
Skrzipek, M.R.:
\newblock Signal recovery by discrete approximation and a {P}rony-like method.
\newblock J. Comput. Appl. Math. \textbf{326} (2017)  193--203

\bibitem{SP19}
Stampfer, K., Plonka, G.:
\newblock The generalized operator-based {P}rony method.
\newblock Constr. Approx., submitted (2019)

\bibitem{Stoica05}
Stoica, P., Moses, R.L.:
\newblock Spectral analysis of signals.
\newblock {Pearson Prentice Hall}, Upper Saddle River, NJ (2005)

\bibitem{Uri13}
Urigen, J., Blu, T., Dragotti, P.:
\newblock {FRI} sampling with arbitrary kernels.
\newblock IEEE Trans. Signal Process. \textbf{61}(21) (2013)  5310--5323

\bibitem{UM14}
Usevich, K., Markovsky, I.:
\newblock Variable projection for affinely structured low-rank approximation in
  weighted 2-norms.
\newblock J. Comput. Appl. Math. \textbf{272} (2014)  430--448

\bibitem{VMB02}
Vetterli, M., Marziliano, P., Blu, T.:
\newblock Sampling signals with finite rate of innovation.
\newblock IEEE Trans. Signal Process. \textbf{50}(6) (2002)  1417--1428

\bibitem{ZP19}
Zhang, R., Plonka, G.:
\newblock Optimal approximation with exponential sums by a maximum likelihood
  modification of {P}rony’s method.
\newblock Adv. Comput. Math. \textbf{45}(3) (2019)  1657--1687

\end{thebibliography}
\bibliographystyle{bibtex/splncs_srt}

%
%
%
%
%
%
%
%
\end{document}